\documentclass[12pt,twoside]{amsart}
\usepackage{amsmath}
\usepackage{amssymb}
\usepackage{amsthm}
\usepackage{latexsym}
\usepackage{amscd}
\usepackage{xypic}
\xyoption{curve}
\usepackage{ifthen}
\usepackage{hyperref}
\usepackage{graphicx}

\newtheorem{theorem}{Theorem}[section]
\newtheorem{lemma}[theorem]{Lemma}
\newtheorem{proposition}[theorem]{Proposition}

\newtheorem{remark}[theorem]{Remark}

\theoremstyle{definition}
\newtheorem{definition}[theorem]{Definition}
\newtheorem{example}[theorem]{Example}

\newcommand {\Aut}{\mathrm{Aut}}
\newcommand {\Ann}{\mathrm{Ann}}
\newcommand {\ord}{\mathrm{ord}}

\newcommand {\gr}{\mathrm{gr}}
\newcommand {\spec}{\mathrm{spec}}

\newcommand {\rk}{\mathrm{rk}}
\newcommand {\sdeg}{\mathrm{sdeg}}
\newcommand {\cdeg}{\mathrm{cdeg}}

\newcommand {\Soc}{\mathrm{Soc}}

\newcommand {\init}{\mathrm{in}}
\newcommand {\ldf}{\mathrm{ldf}}
\newcommand {\tdf}{\mathrm{tdf}}

\newcommand {\Hilb}{\mathcal{H}\kern -0.25ex{\mathit ilb\/}}

\newcommand {\Obs}{\mathcal{O}\kern -0.25ex{\mathit bs\/}}
\newcommand {\Unobs}{\mathcal{U}\kern -0.25ex{\mathit nobs\/}}

\newcommand {\fM}{\mathfrak{M}}

\newcommand {\cU}{\mathcal{U}}

\newcommand {\bN}{\mathbb{N}}

\newcommand {\bP}{\mathbb{P}}
\newcommand {\bA}{\mathbb{A}}

\def\p#1{{\bP^{#1}_{k}}}
\def\a#1{{{\bA}^{#1}_{k}}}

\def\ga#1{{{\accent"12 #1}}}

\def\theroman{\roman{enumi}}
\def\rostere{\begin{list}{\hskip-2truecm(\theroman)}{\usecounter{enumi}}}
\def\endrostere{\end{list}}

\usepackage{xcolor}

\title[$2$--stretched Gorenstein algebras]{A structure theorem\\
 for $2$--stretched Gorenstein algebras}


\author[Gianfranco Casnati, Roberto Notari]{Gianfranco Casnati, Roberto Notar}
\thanks{Both the authors are members of GNSAGA group of INdAM. They have been supported by the framework of PRIN 2010-2011 \lq Geometria sulle variet\ga a algebriche\rq, cofinanced by MIUR}

\begin{document}

\begin{abstract}

In this paper we study the isomorphism classes of local, Artinian, Gorenstein $k$--algebras $A$ whose maximal ideal $\frak M$ satisfies $\dim_k(\fM^3/\fM^4)=1$ by means of Macaulay's inverse system generalizing a recent result by J. Elias and M.E. Rossi. Then we use such results in order to complete the description of the singular locus of the Gorenstein locus of $\Hilb_{11}(\p n)$.
\end{abstract}


\maketitle

\section{Introduction and notation}\label{sIntrNot}

Throughout this paper, a $k$--algebra is an associative, commutative  and unitary algebra over an algebraically closed field $k$ of characteristic $0$.

The study of Artinian $k$--algebras is a classical topic in commutative algebra. It is well--known that each Artinian $k$--algebra is a direct sum of local ones, hence one can restrict its attention to the local case.

Two important invariants of each local, Artinian $k$-algebra $A$ are its
dimension $d:=\dim_k(A)$ as $k$--vector space and the Hilbert function $H_A$ of $A$, i.e. the Hilbert function of the associated graded ring
$$
\gr(A):=\bigoplus_{i=0}^\infty\fM^i/\fM^{i+1},
$$
$\fM$ being the maximal ideal of $A$.

When $d\le 6$ some authors classified such algebras $A$ in terms of $H_A$ (e.g., see \cite{Ma1}, \cite{Ma2}, \cite{Po1}). As $d$
increases the picture becomes much more complicated and not easy
to handle with the same methods (see
\cite{C--E--V--V},\cite{Po2} and the references therein), unless
we introduce some extra technical hypothesis. E.g., if we restrict to Gorenstein algebras, then a complete classification in terms of $H_A$ is
available up to $d\le9$ (see \cite{Cs}). However, once more, as $d$ increases, it is not possible to achieve the complete classification of such algebras. For example, it is not possible to classify in the above sense algebras $A$ with $H_A=(1,4,4,1)$. At the same time, researchers have focused on some interesting classes of local, Artinian, Gorenstein $k$--algebras (see \cite{Sa1}, \cite{Ia1}, \cite{E--V1}, \cite{E--V2}, \cite{E--R1}, \cite{E--R2}).

Nevertheless it is possible to prove a general structure result making use of Macaulay's correspondence. Each local, Artinian, Gorenstein $k$--algebra $A$
can be represented as a quotient of the form
$k[[x_1,\dots,x_n]]/J$ for a suitable ideal
$J\subseteq(x_1,\dots,x_n)^2$. If we look at
$k[[x_1,\dots,x_n]]$ as acting on $k[y_1,\dots,y_n]$ via
derivation (i.e. we identify $x_i$ with $\partial/\partial y_i$,
$i=1,\dots,n$), then $J=\Ann(F)$ for
a suitable  $F\in k[y_1,\dots,y_n]$ whose
degree $s$ is exactly the maximum integer $s$ such that ${\frak
M}^s\ne0$, the so called socle degree of $A$, $\sdeg(A)$.

The main result of \cite{E--R1} is that such an $F$ can be chosen homogeneous when the algebras $A$ satisfy $H_A=(1,n,n,1)$. Hence the classification of such algebras  is actually strictly related to the classification up to projectivities of cubic hypersurfaces in $\p{n-1}$.

In the present paper we extend such a result to Artinian, Gorenstein algebras $A$ with $H_A=(1,n,m,1,\dots)$ in Section \ref{sTail} proving the following theorem (see Theorem \ref{tStruct}).

\medbreak
\noindent
{\bf Theorem A.}
{\it Let $A$ be a local, Artinian, Gorenstein algebra. Then $n=H_A(1)$, $m=H_A(2)$, $1=H_A(3)=\dots=H_A(s)$, $s=\sdeg(A)$ if and only if
$$
A\cong k[[x_1,\dots,x_n]]/\Ann(F)
$$
where $F:=y_1^s+F_3+\sum_{j=m+1}^ny_j^2$,  $F_3\in k[y_1,\dots,y_m]$ is a cubic form, $x_1^2\circ F_3=0$ and $x_2\circ F_{3},\dots, x_{m}\circ F_{3}$ are linearly independent.}
\medbreak

Besides the intrinsic interest of the description of local, Artinian algebras of dimension $d$, their study is also important for the characterization of the singular locus of the Hilbert scheme $\Hilb_d(\p n)$. In some recent papers (see \cite{C--N1}, for the case $d\le9$, and \cite{C--N2}, for the case $d=10$) we dealed with such a problem, restricting our attention to the Gorenstein locus $\Hilb_d^G(\p n)\subseteq \Hilb_d(\p n)$, i.e. the locus of schemes $X\cong\spec(A)$ where $A$ is a Gorenstein algebra. In particular, in \cite{C--N2} the structure theorem of \cite{E--R1} has been an helpful tool for such a description.

In \cite{C--N3} a similar analysis has been carried out in the case $d=11$. In that paper we were also able to deal with the singular nature of all $X\in\Hilb_{11}^G(\p n)$, but those ones isomorphic to $\spec(A)$ where $A$ is a local, Artinian, Gorenstein algebra with $H_A=(1,4,4,1,1)$. The second main result of the present paper is the complete description of such kind of $X$ from such a viewpoint: such a description rests on Theorem A above. The main results of Section \ref{sObstructed} can be summarized in the following Theorem.

\medbreak
\noindent
{\bf Theorem B.}
{\it Let $A$ be a local, Artinian, Gorenstein algebra with $H_A=(1,4,4,1,1)$. Then $X:=\spec(A)\in \Hilb_{11}^G(\p n)$ is obstructed if, and only if, $A\cong S[4]/\Ann(F)$ where $F$ is in the following list:
\begin{enumerate}
\item $y_1^4+y_1(b_0 y_2^2 +  b_2 y_3^2 + b_5 y_4^2)+y_2^3 + y_3^3 + y_4^3$, $b_0,b_2,b_5\in k$;
\item $y_1^4+y_1(b_0 y_2^2 +  b_2 y_3^2 + 2b_4 y_3 y_4 )+y_2^3 + y_3^2y_4$, $b_0,b_2,b_4\in k$;
\item $y_1^4+y_1(b_0 y_2^2 + 2 b_1 y_2 y_3 + b_2 y_3^2 + 2b_0y_3 y_4)+y_2^2y_3 + y_3^2y_4$, $b_0,b_2,b_3\in k$;
\item $y_1^4+y_1(b_0 y_2^2 + 2 b_1 y_2 y_3 + b_2 y_3^2 + 2 b_3 y_2 y_4 + 2
b_4 y_3 y_4 + b_5 y_4^2)+y_3^2y_4-y_3^2y_4$, $b_0,\dots,b_5\in k$, $-b_1^2+b_0b_2-b_1b_3-b_3^2+b_0b_4+b_0b_5=0, (b_0,b_1,b_3)\not=(0,0,0)$;
\item $y_1^4+y_1(b_0 y_2^2 + 2 b_1 y_2 y_3 + b_2 y_3^2 + 2 b_3 y_2 y_4 + 2
b_4 y_3 y_4 + b_5 y_4^2)+y_3^2y_4$, $b_0,\dots,b_5\in k$, $b_1^2-b_0b_2=0, (b_0,b_1,b_3)\not=(0,0,0)$;
\item $y_1^4+y_1(b_0 y_2^2 + 2 b_1 y_2 y_3 + b_2 y_3^2 + 2 b_3 y_2 y_4 + 2
b_4 y_3 y_4 + b_5 y_4^2)+y_3^3$, $b_0,\dots,b_5\in k$ and $\rk(M_b)\geq2$ where
$$
M_b:=\left(\begin{array}{ccc}
b_0 & b_1 & b_3 \\
b_1 & b_2 & b_4 \\
b_3 & b_4 & b_5
\end{array}
\right).
$$
\end{enumerate}}
\medbreak

The authors want to express their thanks to Joachim Jelisiejew and Pedro Macias Marques for some instructing conversation and helpful remarks.

\subsection{Notation}
In what follows $k$ is an algebraically closed field of
characteristic $0$. A $k$--algebra is  an associative, commutative
and unitary algebra over $k$. For each $N\in\bN$ we set
$S[N]:=k[[x_1,\dots,x_N]]$ and $P[N]:=k[y_1,\dots,y_N]$. We denote
by $S[N]_q$ (resp. $P[N]_q$) the homogeneous component of degree
$q$ of such a graded $k$--algebra. Let $S[N]_{\le
q}:=\bigoplus_{i=1}^qS[N]_i$ (resp. $P[n]_{\le
q}:=\bigoplus_{i=1}^qP[n]_i$). Finally we set
$S[n]_+:=(x_1,\dots,x_n)\subseteq S[n]$, the unique maximal ideal
of $S[N]$.

A local ring $R$ is Gorenstein if its injective
dimension as $R$--module is finite. An arbitrary ring $R$ is called Gorenstein if $R_{\frak M}$ is Gorenstein for every maximal ideal ${\frak M}\subseteq R$. A scheme $X$ is Gorenstein if and only if for each point $x\in X$ the ring ${\mathcal O}_{X,x}$ is Gorenstein.

For each numerical polynomial $p(t)\in{\Bbb Q}[t]$, we denote by
$\Hilb_{p(t)}(\p N)$ the Hilbert scheme of closed subschemes of
$\p N$ with Hilbert polynomial $p(t)$. With abuse of notation we
will denote by the same symbol both a point in $\Hilb_{p(t)}(\p
N)$ and the corresponding subscheme of $\p N$. We denote by $\Hilb_{p(t)}^G(\p N)$ the locus of points
representing Gorenstein schemes.

If $\gamma:=(\gamma_1,\dots,\gamma_N)\in{\Bbb N}^{N}$ is a
multi--index, then we set
$t^\gamma:=t_1^{\gamma_1}\dots t_N^{\gamma_N}\in k[t_1,\dots,t_N]$

For all the other notations and results we refer to \cite{Ha}.

\section{Some facts on Macaulay's correspondence}\label{sMacCor}
Let $A$ be a local, Artinian $k$--algebra with maximal ideal $\fM$. We know that
$$
A\cong  S[n] /J
$$
for a suitable ideal $J\subseteq S[n]_+^2\subseteq  S[n]$, where $n:=H_A(1)$. Recall that the {\sl socle degree} $\sdeg(A)$ of $A$ is the greatest integer $s$ such that $\fM^s\ne0$.

We have an action of $S[n]$ over
$P[n]$ given by partial derivation defined by identifying
$x_i$ with $ \partial/ \partial {y_i}$.  Hence
$$
x^{\alpha}\circ y^{\beta}:=\left\lbrace\begin{array}{ll}
\alpha!{\beta\choose\alpha}y^{\beta-\alpha}\qquad&\text{if $\beta\ge\alpha$,}\\
0\qquad&\text{if $\beta\not\ge\alpha$.}
 \end{array}\right.
$$
Such an action endows  $ P[n] $ with a structure of module over $ S[n] $. If $J\subseteq  S[n] $ is an ideal and  $M\subseteq  P[n] $ is a $ S[n] $--submodule we set
\begin{gather*}
J^\perp:=\{\ F\in  P[n] \ \vert\ g\circ F=0,\ \forall g\in J\ \},\\
\Ann(M):=\{\ g\in  S[n] \ \vert\ g\circ F=0,\ \forall F\in M\ \}.
\end{gather*}

For the following results see e.g. \cite{Em}, \cite{Ia2} and the
references therein. Macaulay's theory of inverse system is based
on the fact that such constructions $J\mapsto J^\perp$ and
$M\mapsto \Ann(M)$ give rise to a reversing inclusion bijection
between ideals $J\subseteq S[n]$ such that $S[n] /J$ is a local,
Artinian $k$--algebra and finitely generated $S[n]$--submodules
$M\subseteq P[n]$. In this bijection Gorenstein algebras $A$ with
$\sdeg(A)=s$ correspond to cyclic $ S[n] $--submodules $\langle
F\rangle_{S[n]}\subseteq P[n] $ generated by a polynomial $F$ of
degree $s$. We simply write $\Ann(F)$ instead of $\Ann(\langle
F\rangle_{S[n]})$.

On the one hand, given a $S[n]$--module $M$, we define
$$
\tdf(M)_q:=\frac{{M}\cap P[n]_{\le q}+P[n]_{\le q-1}}{P[n]_{\le q-1}}
$$
where $P[n]_{\le q}:=\bigoplus_{i=0}^qP[n]_i$, and
$\tdf(M):=\bigoplus_{q=0}^\infty\tdf(M)_q$. $\tdf(M)$ can be
interpreted as the $S[n]$--submodule of $P[n]$ generated by the
top degree forms of all the polynomials in $M$.

On the other hand, for each  $f\in S[n]$, the lower degree of monomials appearing with non--zero
coefficient in the minimal representation of $f$ is called {\sl
the order of $f$}\/ and it is denoted by $\ord(f)$. If
$f=\sum_{i=\ord(f)}^{\infty}f_i$, $f_i \in S[n]_i$ then
$f_{\ord(f)}$ is called {\sl the lower degree form of $f$}\/. It
will be denoted in what follows with $\ldf(f)$.

If $f\in J$, then $\ord(f)\ge2$. The {\sl lower degree form ideal} $\ldf(J)$ associated to $J$ is
$$
\ldf(J):=(\ldf(f)\vert f\in J)\subseteq  S[n] .
$$

We have $\ldf(\Ann(M))=\Ann(\tdf(M))$ (see \cite{Em}: see also
\cite{E--R1}, Formulas (2) and (3)) whence
$$
\gr(S[n]/\Ann(M))\cong S[n]/\ldf(\Ann(M))\cong S[n]/\Ann(\tdf(M)).
$$
Thus
\begin{equation}
\label{DerMod}
H_{S[n]/\Ann(M)}(q)=\dim_k(\tdf(M)_q).
\end{equation}
We say that $M$ is {\sl non--degenerate}\/ if
$H_{S[n]/\Ann(M)}(1)=\dim_k(\tdf(M)_1)=n$, i.e. if and only if
the classes of $y_1,\dots,y_n$ are in $\tdf(M)$. If $M=\langle F
\rangle_{S[n]}$, then we write $\tdf(F)$ instead of $\tdf(M)$.

Let $A$ be Gorenstein with $s:=\sdeg(A)$, so that
$\Soc(A)=\fM^s\cong k$. In particular $A\cong  S[n] /\Ann(F)$,
where $F:=\sum_{i=0}^sF_i$, $F_i\in P[n]_i$.
For each $h\ge0$ we set $F_{\ge h}:=\sum_{i=h}^sF_i$ (hence $F_s=F_{\ge s}$). 

Trivially, we can
always assume that $F_0=0$. It is easy to check that,
for given $J=\Ann(F)$, it also holds $ J=\Ann(F + \sigma \circ F)
$ for every $ \sigma \in S[n].$ Hence, it makes sense to look for
an easier polynomial $ G $ such that $\Ann(G) = \Ann(F)$.

{\begin{lemma}
Let $F,\widehat{F}\in P[n]$ be such that $F-\widehat{F}\in P[n]_{\le1}$. If $ \Ann(F) \subseteq
S[n]_+^2,$ then $\Ann(F) = \Ann(\widehat{F})$.
\end{lemma}}
\begin{proof} 
{From $\Ann(F)\subseteq S[n]_+^2$, it follows
$\Ann(F) \subseteq \Ann(\widehat{F})$. In fact, $ \sigma \circ H =
0 $ for each $ H \in P[n]_{\leq 1} $ and every $ \sigma $ of order
at least $ 2.$ The same argument shows that each $ \sigma \in
\Ann(\widehat{F}) $ of order $ \geq 2 $ belongs to $ \Ann(F).$}

{Assume by contradiction that
$\sigma\in\Ann(\widehat{F})\setminus\Ann(F)$. Then, $ \sigma $ has
order $ 1,$ and $ \sigma \circ F = \sigma \circ F_1 = \lambda
\not= 0$. By degree reasons, $ \lambda \in k$. Let $\tau\in S[n]$
be of order $s$ (thus at least $2$) such that $\tau\circ
F_s=-\lambda$. It follows that $\sigma+\tau\in\Ann(F)$ has order
$1$, and so we get a contradiction because $ \Ann(F)\subseteq
S[n]_+^2$. Hence $\Ann(F) = \Ann(\widehat{F})$ as claimed.}
\end{proof}

In particular, $J\subseteq S[n]_+^2$,
then we can assume that $J=\Ann(F)$ where $F=F_{\ge2}$. We will always make such an assumption in what follows.

We have a filtration with proper ideals (see \cite{Ia2}) of $\gr(A)\cong S[n]/\ldf(\Ann(F))$
$$
C_A(0):=\gr(A)\supset C_A(1)\supseteq C_A(2)\supseteq \dots\supseteq C_A(s-2)\supseteq C_A(s-1):=0.
$$
{Via the epimorphism $S[n]\twoheadrightarrow \gr(A)$ we obtain an induced filtration
$$
\widehat{C}_A(0):=S[n]\supset \widehat{C}_A(1)\supseteq \widehat{C}_A(2)\supseteq \dots\supseteq \widehat{C}_A(s-2)\supseteq \widehat{C}_A(s-1):=\ldf(\Ann(F)).
$$}

The quotients $Q_A(a):=C_A(a)/C_A(a+1)\cong  \widehat{C}_A(a)/ \widehat{C}_A(a+1)$ are reflexive graded
$\gr(A)$--modules whose Hilbert function is symmetric around  $(s-a)/2$. In general $\gr(A)$ is no more Gorenstein, but the first quotient
\begin{equation}
\label{GorQuot}
G(A):=Q_A(0)\cong S[n] /\Ann(F_s)
\end{equation}
is characterized by the property of being the unique (up to
isomorphism) graded Gorenstein quotient  $k$--algebra of $\gr(A)$
with the same socle degree. The Hilbert function of $A$ satisfies
\begin{equation}
\label{GorDec}
H_A(i)=H_{\gr(A)}(i)=\sum_{a=0}^{s-2}H_{Q_A(a)}(i),\qquad i\ge0.
\end{equation}
Since $H_A(0)=H_{G(A)}(0)=1$, it follows that if $a\ge1$, then $Q_A(a)_0=0$, whence $Q_A(a)_i=0$ when $i\ge s-a$ (see \cite{Ia2}) for the same values of $a$.

{Moreover
$$
H_{\gr(A)/C_A(a+1)}(i)=H_{S[n]/\widehat{C}_A(a+1)}(i)=\sum_{\alpha=0}^{a}H_{Q_A(\alpha)}(i),\qquad i\ge0.
$$
We set
$$
f_h:=\sum_{\alpha=0}^{s-h}H_{Q_A(\alpha)}(1)=H_{S[n]/\widehat{C}_A(s-h+1)}(1)=H_{\gr(A)/C_A(s-h+1)}(1)
$$}
(so that $n=H_A(1)=f_{2}$). 

Finally we introduce the following new invariant.

\begin{definition}
\label{dCapital}
Let $A$ be a local, Artinian $k$--algebra with maximal ideal $\fM$ and $s:=\sdeg(A)$. The {\sl capital degree}\/, $\cdeg(A)$, of $A$ is defined as the maximum integer $i$, if any, such that $H_A(i)>1$, $0$ otherwise. If $c=\cdeg(A)$ we also say that $A$ is a $c$--stretched algebra (simply stretched if $c\le 1$)
\end{definition}

By definition $\cdeg(A)\ge0$ and $\cdeg(A)\le \sdeg(A)$: if $A$ is Gorenstein, then we also have $\cdeg(A)< \sdeg(A)$.

We recall some facts about algebras $A$ with $\cdeg(A)\le 1$. Trivially, when $\cdeg(A)=0$, then $A\cong S[1]/(x_1^{s+1})$ (recall that $S[1]:=k[[x_1]]$).

Stretched algebras have been completely classified in \cite{Sa1}, with a particular attention to the Gorenstein case. There are many results about $2$--stretched Gorenstein algebras. A complete description of such algebras when $\sdeg(A)=3$ can be found in \cite{E--R1}. $2$--stretched algebras with $H_A(2)=2$ have been examined in several papers (see e.g. \cite{E--V2} or \cite{C--N1}). A complete classification in the case $\sdeg(A)\ge4$ and $H_A(2)=3$ can be found in \cite{C--N1}. Some very partial results are known when $\cdeg(A)=2$, $\sdeg(A)=4$ and $H_A(2)=4$ (see  \cite{C--N2}).

\section{On the homogeneous summands of the apolar polynomial}\label{sPol}
Let $A\cong S[n]/J$ be an Artinian, Gorenstein $k$--algebra where $J=\Ann(F)$ for a suitable $F=F_{\ge 2}\in P[n]$. Such a polynomial strongly depends on the representation of $A$ as a quotient of $S[n]$.

For reader's benefit we recall (see \cite{Ia2}, Theorems 5.3A and 5.3B) in this section that it is always possible to choose a system of generators of $S[n]_+$ such that the $F$ satisfies $F_i\in P[f_i]$. If $A\cong S[n]/\Ann(F)$, then such a property is not authomatically satisfied by $F$, due to the possible existence of {\sl exotic summands} in the homogeneous decomposition of $F$ as the following well--known example shows.

\begin{example}
Let $J:=(x_2^2,20x_1^2x_2-x_1^4)\subseteq S[2]$, then $J^\perp=\langle F\rangle_{S[2]}$, where $F:=y_1^5+y_1^3y_2$. In particular
$$
F_5=y_1^5,\qquad F_4=y_1^3y_2,\qquad F_3=F_2=0.
$$
We have
$\tdf(F)_5=\langle y_1^5\rangle$, $\tdf(F)_4=\langle y_1^4\rangle$, $\tdf(F)_3=\langle y_1^3\rangle$, $\tdf(F)_2=\langle y_1^2,y_1y_2\rangle$, $\tdf(F)_1=\langle y_1,y_2\rangle$, thus relation \eqref{DerMod} implies $H_A=(1,2,2,1,1,1)$. Thus $H_{G(A)}=(1,1,1,1,1,1)$, $H_{Q(1)}=(0,0,0,0,0,0)$, $H_{Q(2)}=(0,1,1,0,0,0)$. In particular $F_4\not\in P[f_4]$ because $f_4=1$ in our case.
\end{example}

\begin{proposition}
\label{pChange}
Let $A$ be a local, Artinian, Gorenstein $k$--algebra. If $n:=H_A(1)$ and $s:=\sdeg(A)$, then
$$
A\cong S[n]/\Ann(F)
$$
where $F:=\sum_{i=2}^sF_i+\sum_{j=f_3+1}^ny_j^2$, $F_i\in P[f_i]_i$, $i\ge3$, and $F_2\in P[f_3]_2$.
\end{proposition}
\begin{proof}
Thanks to the aforementioned Theorems 5.3A and 5.3B of \cite{Ia2} (see also the thesis \cite{Jel}: in particular see Theorem 4.38 where an expanded version of the proof is provided), we know the existence of a representation $A\cong S[n]/\Ann(\sum_{i=2}^sF_i)$ with $F_i\in P[f_i]_i$. Now we prove that $F_2$ can be actually written as the sum of something in $P[f_3]$ plus $\sum_{f_3+1}^ny_i^2$.

Now we prove that we can make a suitable linear change of $y_1,\dots,y_n$ in such a way that the linear space generated by $y_1,\dots,y_{f_3}$ remains unchanged and such that the homogeneous part of degree $2$ of $F$ is $F_2+\sum_{j=f_3+1}^ny_j^2$ where $F_2\in P[f_3]$. First of all, up to a suitable linear transformation of the variables $y_{f_3+1},\dots,y_n$ we can assume that such an homogeneous part of $F$ is
$$
\sum_{j={f_3}+1}^{n}\lambda_jy_j^2+Q
$$
where $\lambda_j\in \{\ 0,1\ \}$ and $Q=\sum_{i=1}^{f_3}\sum_{j=i}^{n}q_{i,j}y_iy_j$.

Since $H_A(1)=n$, we know that the classes of $y_1,\dots,y_{n}$ are in $\tdf( F)$. It follows that we have relations of the form
$$
\sum_{i=1}^{n} u_i(x_i\circ F)+\text{linear combination of derivatives of $F$ of order at least $2$}=y_j+\text{constant}.
$$
Since $F_i\in  P[{f_i}]_i\subseteq P[{f_3}]_i$, $i\ge 3$, we deduce $x_j\circ F=2\lambda_jy_j+x_j\circ Q$.

The only derivatives of $F$ of degree $s-1$ are $x_j\circ F$, $j=1,\dots,f_{s-1}$. They are linearly independent because $\dim_k(\tdf(F)_{s-1})=H_A(s-1)=f_{s-1}$.  It follows that $u_j=0$, $j=1,\dots,f_s$. Since no derivatives of order $2$ contain $y_j$, $j\ge {f_3}+1$ (recall that $F_i\in P[{f_3}]_i$, $i\ge3$), it also follows that the linear combination of such derivatives in the first member of the above equality must be  a costant. We conclude that
$$
\sum_{i={f_3}+1}^{n} u_i(x_i\circ F)=y_j+\text{constant}.
$$
We deduce that necessarily $\lambda_j=1$, $j\ge {f_3}+1$ (recall that $\lambda_j\in\{\ 0,1\ \}$ and $x_j\circ Q\in P[{f_3}]$): up to a suitable linear transformation in the variables $y_1,\dots,y_{n}$ fixing $y_1,\dots,y_{f_3}$ we can finally assume that $F_2:=Q\in P[{f_3}]_2$.

The proof of the statement is complete.
\end{proof}

\section{The structure theorem}
\label{sTail}
Now we turn our attention to the case $\cdeg(A)=2$. Let $H_A=(1,n,m,1,\dots,1)$.
We first prove the following preparatory lemma improving Proposition \ref{pChange} in such a particular case.

\begin{lemma}
\label{lStruct}
Let $A$ be a local, Artinian, Gorenstein $2$--stretched $k$--algebra. If $n:=H_A(1)$, $m:=H_A(2)$, $s:=\sdeg(A)$, then
$$
A\cong S[n]/\Ann(F),
$$
where $F:=y_1^s+F_3+F_2+\sum_{j=m+1}^ny_j^2$, $F_i\in P[m]_i$, $x_1^2\circ F_3=x_1^2\circ F_2=0$ and $x_2\circ F_{3},\dots, x_{m}\circ F_{3}$ are linearly independent.
\end{lemma}
\begin{proof}
If $m=1$ (in particular if $n=1$), then the statement is trivial. Thus we can assume $n,m\ge2$, whence $s\ge3$. 

In Decomposition \eqref{GorDec} we have
$$
H_{Q(i)}=\left\lbrace\begin{array}{ll}
(1,1,1,1,\dots,1)\qquad&\text{if $i=0$,}\\
(0,0,0,0,\dots,0)\qquad&\text{if $i=1,\dots,s-4$,}\\
(0,m-1,m-1,0,\dots,0)\qquad&\text{if $i=s-3$,}\\
(0,n-m,0,0,\dots,0)\qquad&\text{if $i=s-2$.}
 \end{array}\right.
$$
It follows that $f_2=n$, $f_3=m$ and $f_4=\dots=f_s=1$. Thanks to Proposition \ref{pChange} there exists an isomorphism $A\cong S[n]/\Ann(F)$, where $F:=\sum_{i=2}^sF_i+\sum_{j=f_3+1}^ny_j^2$, $F_2\in P[m]_2$, $F_3\in P[m]_3$ and $F_i\in P[1]_i$, $i=4,\dots,s$. In particular if $s\ge4$, then $G(A)\cong S[1]/(y_1^s)$, thus we can assume that $F_s=y_1^s$ thanks to Formula \eqref{GorQuot}.  

{If $s=3$ then $F=\widetilde{F}_3+F_2$. Ut to a linear change of variables in $S[n]$ we can assume that the coefficient of $y_1^3$ in $\widetilde{F}_3$ is $1$, so that $\widetilde{F}_3=y_s^3+F_3$ with $x_1^3\circ {F}_3=0 $ again. }

Now let
$$
F_3=\sum_{{\alpha\in \bN^m}\atop {\vert\alpha\vert=3} }\frac{s!}{\alpha!}u_\alpha y^\alpha,
$$
so that $x^\alpha-u_\alpha
x_1^s\in\Ann(F)$ when $\alpha\in \bN^m$ and $\vert\alpha\vert=3$. 

{Let us consider the automorphism $\varphi$ of $S[{n}]$ defined by
$$
\varphi({x}_j)=\widehat{x}_j:=\left\lbrace\begin{array}{ll}
x_1\qquad&\text{if $ j\ne2,\dots,m$,}\\
x_{j}-u_{2e_1+e_j}x_1^{s-2}\qquad&\text{if $ j=2,\dots,m$.}
 \end{array}\right.
$$
We have an isomorphism
$$
S[{n}]/\varphi^{-1}(\Ann(F))\cong S[{n}]/\Ann(F)\cong A.
$$
We conclude the existence of  $\widehat{F}\in P[{n}]$ such that $\varphi^{-1}(\Ann(F))=\Ann(\widehat{F})$. Let $\widehat{F}=\sum_{i=2}^s\widehat{F}_i$. }

{Due to the definition of $\varphi$, we have that $\widehat{x}^\alpha\in\Ann(F)$ if either $\vert\alpha\vert\ge4$ and $\alpha\ne\vert\alpha\vert e_1$, or if $\widehat{x}^\alpha$ does not contain $x_1,\dots,x_m$ and $\vert\alpha\vert\ge3$, or if $\alpha=2e_1+e_j$, $j=2,\dots,m$. It follows that $x^\alpha\in\varphi^{-1}(\Ann(F))=\Ann(\widehat{F})$ in the same ranges.}

{The first condition implies that we can still assume $\widehat{F}_i\in P[1]$, $i=4,\dots,s$. The second that $\widehat{F}_3\in P[m]$. The third condition implies $x_1^2\circ \widehat{F}_3\in P[1]_1$. A similar argument with monomials of degree $2$ shows that the degree $2$ component of $\widehat{F}$ can be decomposed as $\widehat{F}_2+\sum_{j=m+1}^ny_j^2$ with $\widehat{F}_2\in P[m]$. }

{Summing to $\sum_{i=2}^s\widehat{F}_i$ a suitable linear combination of the derivatives $x_1^i\circ \widehat{F}$ we can finally assume that
$A\cong S[n]/\Ann(\widehat{F})$ where $\widehat{F}:=y_1^s+\widehat{F}_3+\widehat{F}_2+\sum_{j=m+1}^ny_j^2$, and $x_1^3\circ \widehat{F}_3=x_1^2\circ \widehat{F}_2=0$. }

Finally $\tdf(\widehat{F})_2$ is generated by the classes of $y_1^2$ and $x_2\circ \widehat{F}_3,\dots,x_m\circ \widehat{F}_3$. Since
$$
\dim(\tdf(\widehat{F})_2)=H_A(2)=m,
$$
we conclude that $x_2\circ \widehat{F}_3,\dots,x_m\circ \widehat{F}_3$ are linearly independent.
\end{proof}

{\begin{remark}
\label{rTail}
The above lemma can be easily generalized to any local, Artinian, Gorenstein $c$--stretched algebra $A$ for each $c$ as follows. If $n:=H_A(1)$, $m:=H_A(2)$, $s:=\sdeg(A)$, then
$$
A\cong S[n]/\Ann(F),
$$
where $F:=y_1^s+\sum_{i=2}^{c+1}F_i+\sum_{j=m+1}^ny_j^2$, $F_i\in P[f_i]_i$, $i\ge3$, $F_2\in P[f_3]_i$, $x_1^c\circ F_{c+1}=x_1^i\circ F_i=0$, $i=3,\dots, c+1$ and $x_2\circ F_{c+1},\dots, x_{m}\circ F_{c+1}$ are linearly independent.
\end{remark}}

Following the method used to deal with algebras $A$ with $\sdeg(A)=3$ (in \cite{E--R1}) and with compressed algebras (in \cite{E--R2}), we will show how to construct, for each local, Artinian, Gorenstein $k$--algebra $A$ with $H_A=(1,n,m,1,\dots,1)$, $\sdeg(A)=s$, a distinguished polynomial $F_3\in P[m]_3$ such that $A\cong S[n]/\Ann(y_1^s+F_3+\sum_{j=m+1}^ny_j^2)$.

Let $F=y_1^s+F_3+F_2+\sum_{j=m+1}^ny_j^2$ be as in the statement of Lemma \ref{lStruct} and set $A:= S[n]/\Ann(F)$. We look for a particular algebra automorphism $\varphi$ of $S[n]$ mapping $\Ann(F)$ to $\Ann(F_{simple})$ where
$$
F_{simple}:=F-F_2=y_1^s+F_3+\sum_{j=m+1}^ny_j^2.
$$

If we set $\varphi(x_i)=z_i$, then $z_1,\dots,z_n$ is a new minimal set of generators of $S[n]_+$. Thus each $\varphi\in\Aut(S[n])$ induces an element in $\Aut({S[n]/S[n]_+^{s+1}})$ for each $s$ that, improperly, we again denote by $\varphi$. Such an algebra is also a finitely generated vector space on $k$: we fix the basis
$$
{\mathcal X}:=\left(x^\alpha \right)_{\alpha\in{\Bbb N}^n,\vert\alpha\vert\le s}
$$
given by the monomials ordered first by increasing degree and then lexicographically. Thus we can identify each element of $\Aut ({S[n]/S[n]_+^{s+1}})$ with a suitable square matrix. It is immediate to check that the dual basis in $P[n]_{\le s}$ with respect to the perfect pairing $\langle\cdot,\cdot\rangle$ is
$$
{\mathcal Y}:=\left(\frac{1}{\alpha!}y^\alpha \right)_{\alpha\in{\Bbb N}^n,\vert\alpha\vert\le s}
$$
still ordered first by increasing degree and then lexicographically.

Since, in our case, $S[n]_+^{s+1}\subseteq \Ann(F)$, it follows that finding the automorphism $\varphi\in\Aut(S[n])$ mapping $\Ann(F)$ to $\Ann(F_{simple})$ is equivalent to finding an automorphism $\widehat{\varphi}\in \Aut({S[n]/S[n]_+^{s+1}})$ mapping $\Ann(F)/S[n]_+^{s+1}$ to $\Ann(F_{simple})/S[n]_+^{s+1}$.

By duality each automorphism $\varphi\in\Aut(S[n])$ mapping $\Ann(F)$ to $\Ann(F_{simple})$ corresponds to an isomorphism
$$
\varphi^{*{}} \colon(S[n]/\Ann(F_{simple}))^*\longrightarrow(S[n]/\Ann(F))^*
$$
which we interpret as an isomorphism of the subspaces $\varphi^{*}\colon \Ann(F_{simple})^\perp\to\Ann(F)^\perp$ of $P[n]_{\le s}\subseteq P[n]$.

As explained in \cite{E--R1} and \cite{E--R2} via such a correspondence, the matrix $M(\varphi^{*})$ associated to $\varphi^{*}$ with respect to the basis ${\mathcal Y}$ is exactly the transpose of the inverse of the matrix $M(\varphi)$ associated to the morphism $\varphi$ with respect to the basis ${\mathcal X}$.

We are now ready to prove the main result of the paper. It is a structure theorem for $2$--stretched algebras.

\begin{theorem}
\label{tStruct}
Let $A$ be a local, Artinian, Gorenstein $k$--algebra. Then $A$ is $2$--stretched  with $n=H_A(1)$, $m=H_A(2)$, $s=\sdeg(A)$ if, and only if,
$$
A\cong S[n]/\Ann(F)
$$
where $F:=y_1^s+F_3+\sum_{j=m+1}^ny_j^2$,  $F_3\in P[m]_3$, $x_1^2\circ F_3=0$ and $x_2\circ F_{3},\dots, x_{m}\circ F_{3}$ are linearly independent.
\end{theorem}
\begin{proof}
In the following we will set $N(h):=\{\ \alpha\in\bN^m\ \vert\ \vert\alpha\vert=h,\ \alpha\ne he_1\ \}$.

The \lq\lq if\rq\rq\ part is easy to prove. Since $F_3\in P[m]_3$ it follows that $\tdf(F)_q=\langle s!y_1^q+q!x_1^{s-q}\circ F_3\rangle$, $3\le q\le s$. Due to Equality \eqref{DerMod}, $H_A(q)=\dim_k(\tdf(F)_q)=1$ in the same range.
Clearly $\langle s!y_1^2+2x_1^{s-2}\circ F_3, x_2\circ F_3,\dots,x_m\circ F_3\rangle\subseteq\tdf(F)_2$. Since $x_2\circ F_{3},\dots, x_{m}\circ F_{3}$ are linearly independent and do not contain $y_1^2$, because $x_1^2\circ F_3=0$, it follows that equality holds, thus again $H_A(2)=\dim_k(\tdf(F)_2)=m$.

Similarly $\langle s!y_1+x_1^{s-1}\circ F_3, x^\gamma\circ F_3,y_{m+1},\dots,y_n\rangle_{\gamma\in N(2)}\subseteq\tdf(F)_1$. Again we actually have an equality. Indeed the only possible new element in $\tdf(F)_1$ could come from a linear combination of $s!y_1^2+2x_1^{s-2}\circ F_3, x_2\circ F_3,\dots,x_m\circ F_3$ when $s=4$. Again the condition $x_1^2\circ F_3$ guaratees that this cannot occur. It follows that
$$
\tdf(F_{\ge 3})_1=\langle s!y_1+2x_1^{s-1}\circ F_3, x^\gamma\circ F_3\rangle_{\gamma\in N(2)}.
$$
Thus the Hilbert function of $B:=S[m]/(F_{\ge 3})$ is $H_B=(1,a,m,1,\dots,1)$. Thanks to Formula \eqref{GorDec} we know that $a\ge m$. Since $F_{\ge 3}\in P[m]$ we necessarily have 
$$
\langle s!y_1+2x_1^{s-1}\circ F_3, x^\gamma\circ F_3\rangle_{\gamma\in N(2)}=\langle y_1,\dots,y_m\rangle,
$$
whence $H_A(1)=n$.

Now we prove the \lq\lq only if\rq\rq\ part. Thanks to Lemma \ref{lStruct} we know that
$$
A\cong S[n]/\Ann(F),
$$
where $F:=y_1^s+F_3+F_2+\sum_{j=m+1}^ny_j^2$, $F_i\in P[m]_i$, $x_1^2\circ F_3=x_1^2\circ F_2=0$ and $x_2\circ F_{3},\dots, x_{m}\circ F_{3}$ are linearly independent. We first examine the case $n=m$: the changes in the case $n>m$ will be listed at the end of the proof.

Imitating the proof of \cite{E--R1}, Theorem 3.3, we look for a suitable automorphism $\varphi\in\Aut(S[n]_{\le s})$ defined as
$$
\varphi(x_j)=x_j+\sum_{\gamma\in N(2)}b_{\gamma,j}x^\gamma.
$$
whose dual morphism transforms $\Ann(F_{simple})^\perp$ to $\Ann(F)^\perp$. The matrix $M(\varphi)$ with respect to $\mathcal Y$ is
\begin{equation}
\label{matrix}
B:=\left(\begin{array}{cccccc}
1&0&0&0&\dots&0\\
0&I_n&0&0&\dots&0\\
0&B(2,1)&I_{n+1\choose2}&0&\dots&0\\
0&0&B(3,2)&I_{n+2\choose3}&\dots&0\\
0&0&B(4,2)&B(4,3)&\dots&0\\
0&0&0&B(5,3)&\dots&0\\
\vdots&\vdots&\vdots&\vdots&\ddots&\vdots\\
0&0&0&*&\dots&I_{n+e-1\choose e}
\end{array}
\right)
\end{equation}
where $I_h$ is the identity matrix of order $h$, $B(i,j)$ are matrices of order  ${n+i-1\choose i}\times {n+j-1\choose j}$ whose entries are forms of degree $i-j$ in the  $b_{\gamma,j}$'s, $0$ is a zero matrix of suitable dimensions.

Set
$$
F_3=\sum_{\alpha\in N(3)}\frac{s!}{\alpha!}u_\alpha y^\alpha,\qquad F_2=\sum_{\beta\in N(2)}\frac{s!}{\beta!}v_\beta y^\beta.
$$
We have $\sum_{i=1}^nu_{2e_1+e_i}y_i=x_1^2\circ F_3=0$, $v_{2e_1}=x_1^2\circ F_2=0$.

We denote by $\Delta$ the $n\times{n+1\choose2}$ matrix whose $t^{th}$ row is the vector of the coordinates of $x_t\circ F_3$ with respect to the basis $\mathcal Y$. We notice that the condition $x_1^2\circ F_3=0$ implies that the first column of $\Delta$ is zero.

Let $0_t$ be the $0$ vector of $k^{\oplus t}$ and let $e$ be the first vector of the canonical basis of $k^{n+s-1\choose s}$ (thus $e$ is the vector of the components of $y_1^s$ with respect to $\mathcal Y$). The component of $F$ and $F_{simple}$ with respect to the basis $\mathcal Y$ of $P[n]_{\le s}$ are respectively
\begin{gather*}
[F]_{\mathcal Y}=(0,0_n,s!v_\beta,s!u_\alpha,0_{n+3\choose 4},\dots,0_{n+s-2\choose s-1},e)_{{\beta\in \bN^m,\vert\beta\vert=2}\atop{\alpha\in \bN^m,\vert\alpha\vert=3}},\\
[F_{simple}]_{\mathcal Y}=(0,0_n,0_{n+1\choose2},s!u_\alpha,0_{n+3\choose 4},\dots,0_{n+s-2\choose s-1},e)_{{\beta\in \bN^m,\vert\beta\vert=2}\atop{\alpha\in \bN^m,\vert\alpha\vert=3}}.
\end{gather*}

By duality we have to look for a $\varphi$ such that
\begin{align}
\label{FormulaV}
(u_\alpha)_{{\alpha\in \bN^m,\vert\alpha\vert=3}} B(3,2)=(v_\beta)_{\beta\in \bN^m,\vert\beta\vert=2}.
\end{align}
Notice that the columns of $B(3,2)$ are exactly the coefficients of the forms of degree $3$ in the products
$$
(x_j+\sum_{\gamma\in \bN^m,\vert\gamma\vert=2}b_{\gamma,j}x^\gamma)(x_h+\sum_{\delta\in \bN^m,\vert\delta\vert=2}b_{\delta,h}x^\delta)
$$
for $j\le h=1,\dots,n$. It follows that the entry on the $\alpha^{th}$ row and on the $(j,h)^{th}$ column is
$$
B(3,2)_{\alpha,(j,h)}=\left\lbrace\begin{array}{ll}
b_{\delta+e_j,h}+b_{\gamma+e_h,j}\qquad&\text{if $\alpha\ge e_j+e_h$,}\\
b_{\delta+e_j,h}\qquad&\text{if $\alpha\ge e_j$,\ $\alpha\not\ge e_h$,}\\
0\qquad&\text{otherwise.}
 \end{array}\right.
$$

Thus the entries of the product $(u_\alpha)_{\alpha\in N(3)}
B(3,2)$ are bihomogeneous forms in the $u_\alpha$'s and
$b_{\gamma,j}$'s. Hence there is a suitable $n{n+1\choose2}\times
{n+1\choose 2}$ matrix $U$ whose coefficients depend on the
$u_\alpha$'s and such that
\begin{align}
\label{FormulaU}
(u_\alpha)_{\alpha\in \bN^m,\vert\alpha\vert=3} B(3,2)=(b_{\gamma,j})_{\gamma\in \bN^m,\vert\gamma\vert=2,j=1,\dots,n}{}^tU
\end{align}
(the $b_{\gamma,j}$'s are ordered first with respect to $\gamma$ and then with respect to $j$).
Thus we obtain from equalities \eqref{FormulaV} and \eqref{FormulaU} the system of linear equations
\begin{align}
\label{system1}
U\ {}^t(b_{\gamma,j})_{\gamma\in \bN^m,\vert\gamma\vert=2,j=1,\dots,n}={}^t(v_\beta)_{\beta\in \bN^m,\vert\beta\vert=2}
\end{align}
in the variables $b_{\gamma,j}$'s. We recall that in  \cite{E--R1}, in the proof of Theorem 3.3, it is proved that $U$ is a lower triangular block matrix
$$
U=\left(\begin{array}{cccccc}
U(1)&*&*&\dots&*&*\\
0&U(2)&*&\dots&*&*\\
0&0&U(3)&\dots&*&*\\
\vdots&\vdots&\vdots&\ddots&\vdots&\vdots\\
0&0&0&\dots&U(n-1)&*\\
0&0&0&\dots&0&U(n)
\end{array}
\right)
$$
where $U(h)$ is a $(n-h+1)\times{n+1\choose 2}$ matrix whose first row is twice the $h^{th}$ row of the matrix $\Delta$ (previously defined as the matrix of the partial derivatives of $F_3$) and the $t^{th}$ row is exactly the $(h+t-1)^{th}$ row of $\Delta$, $t=2,\dots,n-h+1$.

Due to the independence of the derivatives  $x_2\circ F_{3},\dots,
x_{m}\circ F_{3}$ it thus follows that the rank of the submatrix
obtained by erasing the first row of $U$ is maximal. Moreover the
constant term of the first equation is $v_{2e_1}$ which is zero,
because $x_1^2\circ F_2=0$. It follows the existence of a solution
the system \eqref{system1} with $b_{\gamma,1}=0$,
$\gamma\in\bN^m$, $\vert\gamma\vert=2$.

In order to extend the above proof also to the case $n>m$ it suffices to change the ordering on $\mathcal X$ and, consequently, on $\mathcal Y$. In this case we fix an order on $\mathcal X$ by taking first all the monomials in $x_1,\dots,x_m$ (ordered first by degree and then lexicographically), and then all the remaining monomials in any order. Thus
$$
M(\varphi)=\left(\begin{array}{cc}
B&0\\
*&B'
\end{array}
\right),
$$
where $B$ is as in (\ref{matrix}) and $B'$ is a suitable matrix whose entries depend on the $b_{\gamma,j}$'s such that $\gamma\ge e_j$ for some $j\ge m+1$. We can thus repeat the above arguments obtaining a system of the form
\begin{align}
\label{system2}
\left(\begin{array}{cc}
U&*\\
0&U'
\end{array}
\right)
\ {}^t(b_{\gamma,j})_{\gamma\in \bN^m,\vert\gamma\vert=2,j=1,\dots,n}={}^t(v_\beta)_{\beta\in \bN^m,\vert\beta\vert=2}
\end{align}
where $U$ is as above and the entries of $U'$ depend on $u_\alpha$ such that $\alpha\ge e_j$ for some $j\ge m+1$. On the one hand, thanks to Lemma \ref{lStruct} we know that such $u_\alpha$ are all zero, i.e. $U'$ is the zero matrix. On the other hand, again by Lemma \ref{lStruct} we know that $v_\beta=0$, for $\beta\ge e_j$ when $j\ge m+1$. We deduce that such a system has again solutions, and, in particular, one of its solutions satisfies $b_{\gamma,j}=0$, $\gamma\in \bN^m$, $\vert\gamma\vert=2$, $j=1,m+1,\dots,n$.
\end{proof}

\section{Obstructedness of a class of algebras}
\label{sObstructed}
In this section we make use of the above structure theorem in
order to deal with the obstructedness of the points in
$\Hilb_{11}^G(\p N)$ corresponding to schemes $X\cong\spec(A)$ where
$A$ is a local, Artin, Gorenstein $k$-algebra with Hilbert
function $H_A=(1,4,4,1,1)$. Thus $A\cong S[4]/J$, where $J$
contains
\begin{equation}
\label{highDegree}
x^\beta,\ x^\alpha,\qquad \beta,\alpha\in \bN^4,\ \vert\beta\vert=4,\ \beta\ne4e_1,\ \vert\alpha\vert=5.
\end{equation}
It follows that $S[4]_+^5\subseteq J$, hence there is a natural isomorphism
$$
\frac{S[4]}{J}\cong\frac{ k[x_1,x_2,x_3,x_4]}{J\cap [x_1,x_2,x_3,x_4]},
$$
inducing a natural epimorphism $k[x_1,x_2,x_3,x_4]\twoheadrightarrow A\cong S[4]/J$, i.e. an embedding $X\subseteq\a4\subseteq\p4$.

In \cite{C--N3} we proved the irreducibility of $\Hilb_{11}^G(\p n)$ studying the locus of singular $X$ such that $X\cong\spec(A)$, where $A$ is not local with $H_A=(1,4,4,1,1)$. Thus a point $X\in \Hilb_{11}(\p n)$ is singular (i.e. the corresponding space is obstructed) if the dimension of the tangent space at $X$ to $\Hilb_{11}(\p n)$ is greater than $\dim(\Hilb_{11}^G(\p n))$.

Obstructedness depends only on the intrinsic structure of $X$ (see \cite{C--N1} and the references therein),
hence only on $A$, we can restrict our attention to the aforementioned embedding in $\p4$ and we simply speak about the obstructedness of the algebra $A$.

Recall that the
tangent space to $\Hilb_{11}^G(\p 4)$ at $X$ is
canonically identified with $H^0\big(X,{\mathcal
N}\big)$ of the global sections of the normal sheaf $\mathcal N$
of $X:=\spec(S[4]/J)\subseteq\a4$. Thus $X$ is unobstructed if and
only if
$$
N_{S[4]/J}:=h^0\big(X,{\mathcal N}\big)=\dim\big(\Hilb_{11}^{G}(\p 4)\big)=44.
$$
In \cite{C--N1} we pointed out that $N_{S[4]/J}=\dim_k(S[4]/J^2)-11$.

Thanks to Theorem \ref{tStruct}, we can assume $J=\Ann(F)$, where
$F=y_1^4+F_3$, $F_3\in P[4]$, $x_1^2\circ F_3=0$ and $x_2\circ
F_3$, $x_3\circ F_3$, $x_4\circ F_3$ are linearly independent.
So,
$$
F=y_1^4+y_1Q+H,
$$
where $Q\in k[y_2,y_3,y_4]_2$ and $H\in k[y_2,y_3,y_4]_3$. Up to a
suitable linear transformation of the variables $y_2,y_3,y_4$,
then $H$ is either $0$ or it can be put in one of the following forms
\begin{equation}
\label{List}
{\begin{gathered}
y_2^3+y_3^3+y_4^3+ty_2y_3y_4,\quad y_2^3+y_3^3+y_2y_3y_4,\quad y_2^3+y_2y_3y_4,\quad y_2y_3y_4,\\
y_2^3 +y_3^2y_4,\quad y_2^2y_3+y_3^2y_4,\quad  y_3^2y_4-y_3y_4^2,\quad  y_3y_4^2,\quad  y_4^3.
\end{gathered}}
\end{equation}
We will examine separately the above cases in what follows computing $\dim_k(S[4]/J^2)$.

Before starting with the description in the different cases we spent a few words on the methods used to perform the computations.

Thanks to Relations \ref{highDegree} we know that $S[4]_+^{10}\subseteq J^2$, thus
$$
\frac{S[4]}{J^2}\cong\frac{ k[x_1,x_2,x_3,x_4]}{J^2\cap [x_1,x_2,x_3,x_4]},
$$
too. In particular when we perform computations, we can always
work in the polynomial ring $k[x_1,x_2,x_3,x_4]$ instead of
$S[4]$. For this reason we can make use of the Computer Algebra
Software Singular \cite{sing} for all the computations in $S[4]$.

Moreover, $ J $ is never homogeneous. We
computed $\init(J) $ and $\init(J^2) $ with respect to
the product term order for which
\rostere
\item $x_4 > x_3 > x_2 > x_1;$
\item the graded reverse lexicographic order on $ x_4, x_3, x_2;$
\item the lexicographic order on $ x_1.$
\endrostere
Hence, for comparing $
x_4^{a_4} x_3^{a_3} x_2^{a_2} x_1^{a_1} $ and $ x_4^{b_4}
x_3^{b_3} x_2^{b_2} x_1^{b_1},$ we first compare $ x_4^{a_4}
x_3^{a_3} x_2^{a_2} $ and $ x_4^{b_4} x_3^{b_3} x_2^{b_2} $ with
respect to degrevlex, and, if they are equal, we compare $
x_1^{a_1} $ and $ x_1^{b_1} $ with respect to lex. For such a
choice, the Hilbert function of $ S[4]/\init(J) $ is $
(1,4,4,1,1),$ while the Hilbert function of $ S[4]/\init(J^2)
$ is the one indicated above.

\begin{proposition}
\label{H=0}
Let $A:=S[4]/\Ann(y_1^4+y_1Q)$ with $Q\in k[y_2,y_3,y_4]_2$ and $H_A=(1,4,4,1,1)$. Then $N_A=49$, hence $A$ is obstructed.
\end{proposition}
\begin{proof}
We know that $x_i\circ F_3$, $i=2,3,4$, are linearly independent because of Theorem \ref{tStruct}. It follows that, up to a linear change of the variables
$y_2,y_3,y_4$, we can assume $Q=y_2^2+y_3^2+y_4^2$. In this case $J=\Ann(F)$ is generated by
$$
x_4^2-x_2^2,\quad x_3^2-x_2^2,\quad x_4x_3,\quad x_4x_2,\quad x_3x_2,\quad 12x_4^2-x_1^3,\quad
x_2x_1^2,\quad x_3x_1^2,\quad x_4x_1^2.
$$
Using the software Singular we obtain $H_{S[4]/J^2}=(1,4,10,20,20,4,1)$, thus $N_{A}=49>44$.
\end{proof}

Now we turn our attention to the case $H\ne0$. It follows that $H$ is one of the polynomials indicated in the list \eqref{List}. We outline the strategy for the computations.

Once $H\ne 0$ is fixed, for each $b:=(b_0,\dots,b_5)\in\a6$, we consider the quadratic form
$$
Q_b:=b_0 y_2^2 + 2 b_1 y_2 y_3 + b_2 y_3^2 + 2 b_3 y_2 y_4 + 2
b_4 y_3 y_4 + b_5 y_4^2
$$
The associated symmetric matrix is the matrix $M_b$ defined in the introduction.

We define $F^{H,b}:=y_1^4+y_1Q_b+H$ and $A^{H,b}:=S[4]/\Ann(F^{H,b})$. Hence $ \Ann(F^{H,b})$ depends on $
x_1, \dots, x_4 $ and $ b_0,\dots,b_5$. Thanks to Theorem \ref{tStruct} and \cite{Em}, Corollary at p. 415, we have a flat family whose base is the open non--empty subset
$$
B_H:=\{\ b\in\a6\ \vert\ \text{$x_i\circ (y_1Q_b+H)$, $i=2,3,4$, are linearly independent}\ \}.
$$
Trivially we have to compute the Hilbert
function of $ \Ann(F^{H,b})^2 $ as a function of $b$.

The following result helps us to simplify the computations in several case.
\begin{lemma}
\label{0-smooth}
Let $H\in k[y_2,y_3,y_4]_3$ be fixed. If $0\in B_H$ and $A^{H,0}$ is unobstructed, then $B_H=\a6$ and for every $b\in\a6$ we have that $A^{H,b}$ is unobstructed.
\end{lemma}
\begin{proof}
Fix $H$ and $b$. Theorem \ref{tStruct} implies that $H_{A^{H,0}}=(1,4,4,1,1)$ because $0\in B_H$. In particular $x_i\circ F^{H,0}_3=x_i\circ H$, $i=2,3,4$ are linearly independent, thus the same is true for $x_i\circ F^{H,tb}_3$, $i=2,3,4$ without restrictions on $t\in k$ and $b\in\a6$, because $y_1$ does not appear in $H$. In particular $B_H=\a6$.

Again Theorem \ref{tStruct} implies that the Hilbert function of $A^{H,tb}$ is $(1,4,4,1,1)$, thus we have a flat family (\cite{Em}, Corollary at p. 415) of deformations of $A^{H,0}$ with base $\a1$. If $t\ne0$, the automorphism of $P[4]$ defined by $(y_1,y_2,y_3,y_4)\mapsto(t^3y_1,t^4y_2,t^4y_3,t^4y_4)$ shows that all the other deformations are isomorphic to $A^{H,b}$. Since $A^{H,0}$ is unobstructed, it follows that the general one, i.e. $A^{H,b}$, is unobstructed too.
\end{proof}

As immediate application of the above Lemma we obtain the following general result.

\begin{proposition}
\label{general}
Let $H$ be either $y_2^3+y_3^3+y_4^3+ty_2y_3y_4$ with $ t(t^3-216)\ne0$, or $y_2^3+y_3^3+y_2y_3y_4$, $y_2^3+y_2y_3y_4$, or $y_2y_3y_4$. Then $B_H=\a6$ and $A^{H,b}$ is unobstructed.
\end{proposition}
\begin{proof}
It is immediate to check that $B_{H}=\a6$. The above cases are handled in a similar way, thus we only deal with the first one. In this case we will write $H_t$ instead of $H$. At a first stage, we look for $ t$ such that $0\in\cU_t$, where $\cU_t\subseteq\a6$ is the open subset of $b$ such that $A^{H_t,b}$ is unobstructed. For such values of $ t,$ we have $B_{H_t}= \a6$ (see Theorem \ref{0-smooth}).

The ideal $J_t:=\Ann(F^{H^t,0})$ is non--minimally generated by
\begin{gather*}
x_1 x_2,\quad
x_1 x_3,\quad
x_1 x_4,\quad
t x_2^2-6 x_3 x_4,\quad
t x_3^2-6 x_2 x_4,\quad
t x_4^2-6 x_2 x_3,\quad
x_1^2 x_2,\quad
x_1 x_2^2,\\
x_1^2 x_3,\quad
x_1 x_2 x_3,\quad
x_2^2 x_3,\quad
x_1 x_3^2,\quad
x_2 x_3^2,\quad
x_1^2 x_4,\quad
x_1 x_2 x_4,\quad
x_2^2 x_4,\quad
x_1 x_3 x_4,\quad
x_3^2 x_4,\\
x_1 x_4^2,\quad
x_2 x_4^2,\quad
x_3 x_4^2,\quad
4 x_2^3-x_1^4,\quad
4 x_3^3-x_1^4,\quad
24 x_2 x_3 x_4-t x_1^4,\quad
4 x_4^3-x_1^4;
\end{gather*}
The unique relation on $ t $ for which $\init(J_t^2)$ changes is $ t(t^3-216) $. Thus $0\not\in\cU_t$ if, and only if, $  t(t^3-216) = 0$.
\end{proof}

\begin{remark}
The form $y_2^3+y_3^3+y_4^3+ty_2y_3y_4$ is the sum of the cubes of three linearly independent linear forms (i.e. it represent a Fermat cubic in the projective plane) if, and only if, $  t(t^3-216) = 0$.
\end{remark}

In the remaining cases the above argument does not work. Anyhow, as it will be evident from the computations below, the generators of $ \Ann(F^{H,b})$, hence of $ \Ann(F^{H,b})^2$, in $ k[b]\otimes_k S[4],$ have homogeneous coefficients in $k[b]$, and this
property still holds when we compute Gr\"obner bases either of $
\Ann(F^{H,b})$, or of $ \Ann(F^{H,b})^2$. Hence
the base for our family is the open non--empty subset of $B_H\subseteq\a6$ where the three derivatives $x_i\circ (y_1Q_b+H)$, $i=2,3,4$, are linearly independent, but all the coefficients describe cones
through the origin.

Now, we focus on $\Ann(F^{H,b})^2$. The Hilbert function
of $ S[4]/\Ann(F^{H,b})^2$ will be constant on an open subset of $
B_H.$ If the Hilbert function of $ S[4]/\Ann(F^{H,b})^2$ changes for
suitable $ b,$ then $\init(\Ann(F^{H,b})^2)$ must change as well. Hence, we compute a Gr\"obner basis of
$\Ann(F^{H,b})^2$ with respect to the term order described above, and we compute, for each monomial of $\init(\Ann(F^{H,b})^2),$ the homogeneous ideals in $ k[b]
$ spanned by its coefficients. We compute the prime ideals $r_*$
associated to such ideals, giving us a set of
level $ 1 $ conditions that can force the Hilbert function of $
\Ann(F^{H,b})^2 $ to change. Of course, we can restrict to the associated
prime ideals because we study set--theoretically the family $S[4]/\Ann(F^{H,b})\to B_H$. By computing the Gr\"obner basis of $r_*+ \Ann(F^{H,b})^2,$
we get a new initial ideal that we study exactly as before,
obtaining a new set of level $ 2 $ prime ideals. Going on with
this strategy, we construct a tree that we analyze leaf by leaf
from the point of view of the Hilbert function.

Before listing the results, we make some remarks on the way we
perform the computations.

The Gr\"obner bases computations can be performed in $
k[x_4, \dots, x_1, b_0, \dots, b_5],$ up to choosing a
product term ordering, with $ 3 $ blocks of variables: $ x_4 > x_3
> x_2, x_1, b_0 > \dots > b_5,$ and degrevlex orders the monomials
in the first and last block, while lex orders the monomials in the
second block.

To compute the Hilbert function of an ideal over the general
element of the variety defined by a chosen prime ideal, we proceed
in the following way: we compute the initial ideal of the ideal
generated by the prime ideal and $\Ann(F^{H,b})^2,$ we forget the monomials
contained in $ k[b],$ and then we reduce the resulting monomial
ideal by setting $b_i=1$ for each $i$. As last step, we compute
the Hilbert function of the ideal we get.

To compute the ideal spanned by the coefficients of a particular
initial monomial $ M \in k[x_1, \dots, x_4],$ we select
the polynomials in the Gr\"obner basis having initial monomial $ M
M' $ with $ M' \in k[b_0, \dots, b_5],$ then we compute
the remainder of every such polynomial modulo $ M,$ and the
difference $ D $ between the polynomial and its remainder. $ D $
has $ M $ as factor, and we call $ d $ the quotient. The computed
$d$'s are the generators of the ideal spanned by the
coefficients. The prime ideals associated to such an ideal are
computed by using the package Primdec in Singular, when it
produces the result (this happened always but a few cases in which
we had to make the computation by hand with ad hoc techniques,
because the size of the ideal was too big for Singular to compute
the result in a reasonable time).

In next proposition, we deal with the cases $H\ne0$ not covered by Proposition \ref{general}.

\begin{proposition}
\label{general}
Let $H=y_2^3+y_3^3+y_4^3$. Then $B_H=\a6$ and $A^{H,b}$ is obstructed if, and only if, $b\in V(b_1,b_3,b_4)\subseteq\a6$.

Let $H=y_2^3 + y_3^2y_4$. Then $B_H=\a6$ and $A^{H,b}$ is obstructed if, and only if, $b\in V(b_1,b_3,b_5)\subseteq\a6$.

Let $H=y_2^2y_3 + y_3^2y_4$. Then $B_H=\a6$ and $A^{H,b}$ is obstructed if, and only if, $b\in V(b_0-b_4,b_3,b_5)\subseteq\a6$.

Let $ H = y_3^2y_4-y_3^2y_4$. Then $B_H=\a6\setminus V(b_0,b_1,b_3)$ and $A^{H,b}$ is obstructed if, and only if, $b\in V(-b_1^2+b_0b_2-b_1b_3-b_3^2+b_0b_4+b_0b_5)\subseteq\a6$.

Let $H=y_3y_4^2 $. Then $B_H=\a6\setminus V(b_0,b_1,b_3)$ and and $A^{H,b}$ is obstructed if, and only if, $b\in V(b_1^2-b_0b_2)\subseteq\a6$.

Let $H=y_4^3$. Then $B_H=\{\ b\in\a6\ \vert\ \rk(M_b)\ge2\ \}$ and $A^{H,b}$ is obstructed for each $b\in B_H$.

In all the aforementioned cases, if $A^{H,b}$ is obstructed, then $N_{A^{H,b}}=49$.
\end{proposition}
\begin{proof}
In each of the above cases we indicate a non--minimal set of generators of the ideal $ \Ann(F^{H,b}) $ and the locus in $B_H$ corresponding
to obstructed ideals. Only in the first case of the list
we report also the tree and
the Hilbert functions on the general element of each subset of $B_H$ where the initial ideal changes.

Let $H=y_2^3+y_3^3+y_4^3$. It is immediate to check that $B_{H}=\a6$. Again we will denote by $\cU$ the open and non--empty subset of $b\in \a6$ such that $A^{H,b}$ is unobstructed. Since $H$ is fixed throughout the whole proof we will simply write $J_b$ instead of $\Ann(F^{H,b})$. $J_b$ is non--minimally generated by
\begin{gather*}
3x_2x_1-b_0p_1-b_1p_2-b_3p_3,\quad 3x_3x_1-b_1p_1-b_2p_2-b_4p_3, \quad 12x_3x_2-b_1x_1^3,\\
3x_4x_1-b_3p_1-b_4p_2-b_5p_3, \quad 12x_4x_2-b_3x_1^3, \quad 12x_4x_3-b_4x_1^3, \quad x_2x_1^2, \\
 12x_2^2x_1-b_0x_1^4, \quad  4x_2^3-x_1^4, \quad x_3x_1^2, \quad 12x_3x_2x_1-b_1x_1^4, \quad x_3x_2^2, \\
  12x_3^2x_1-b_2x_1^4, \quad x_3^2x_2, \quad 4x_3^3-x_1^4, \quad x_4x_1^2, \quad 12x_4x_2x_1-b_3x_1^4, \quad x_4x_2^2, \\
   12x_4x_3x_1-b_4x_1^4, \quad x_4x_3x_2, \quad x_4x_3^2, \quad 12x_4^2x_1-b_5x_1^4, \quad  x_4^2x_2, \quad x_4^2x_3, \quad 4x_4^3-x_1^4
\end{gather*}
where
$$
p_1=x_2^2-b_0x_1^3/12,\quad p_2=x_3^2-b_2x_1^3/12,\quad p_3=x_4^2-b_5x_1^3/12,
$$
and by the monomials in \eqref{highDegree}. 

We know $0\not\in\cU$, with $ N_{S[4]/J_0}=49$. We have that $\init( J^2_b) $ could change only if $b$ is in the variety $V(r_*)\subseteq\a6$ where $r_*$ is one of the following ideals:
\begin{itemize}
\item  $ r_1 = ( b_1, b_3, b_4 );$
\item $ r_2 = ( b_3, b_4, b_5 );$
\item  $ r_3 = (  b_3 );$
\item  $ r_4 = (  b_1 );$
\item  $ r_5 = (  b_1, b_3 );$
\item  $ r_6 = (  b_1, b_4 );$
\item  $ r_7 = (  b_3, b_4 );$
\item  $ r_8 = (  b_0, b_1, b_3 );$
\item  $ r_9 = (  -b_1^2+b_0b_2 );$
\item  $ r_{10} = (  -b_1b_3+b_0b_4 );$
\item  $ r_{11} = (  b_1b_2^2b_3^2-2b_1^2b_2b_3b_4+b_1^3b_4^2-b_3^3b_4^2+2b_1b_3^2b_4b_5-b_1^2b_3b_5^2 );$
\item  $ r_{12} = (  -b_2b_3^2+2b_1b_3b_4-b_0b_4^2-b_1^2b_5+b_0b_2b_5 );$
\item  $ r_{13} = (  b_1, b_2, b_4 );$
\item  $ r_{14} = (  -b_2b_3+b_1b_4, -b_1b_3+b_0b_4, -b_1^2+b_0b_2 );$
\item  $ r_{15} = (  -b_4^2+b_2b_5, -b_3b_4+b_1b_5, -b_2b_3+b_1b_4 );$
\item  $ r_{16} = (  -b_3b_4+b_1b_5, -b_3^2+b_0b_5, -b_1b_3+b_0b_4 );$
\item  $ r_{17} = (  -b_4^2+b_2b_5, -b_3b_4+b_1b_5, -b_2b_3+b_1b_4, -b_3^2+b_0b_5, -b_1b_3+b_0b_4, -b_1^2+b_0b_2 ).$
\end{itemize}

Those ideals are the first level conditions in the tree we are
going to construct. Now, we impose the conditions one at a time,
i.e. we compute a Gr\"obner basis of  $ r_* + J^2_b,$ (i.e.
$J^2_b$ restricted to $r_*$) and we analyze its initial ideal. We
start from the ones of larger codimension and smaller degree.

So, we consider first codimension $ 3 $ ideals. Let us first look at $ r_1$. If we denote by $r_{1,1},\dots,r_{1,N}$ the second order condition, we find $b\in V(r_1)\setminus\bigcup_{j=1}^N V_{r_{1,j}}$ which is in $\cU$ or, in other words, such that $A^{H,b}$ is obstructed. Since $V(r_1)$ is irreducible, it follows that $V(r_1)\cap\cU=\emptyset$ due to the semicontinuity of the Hilbert function of $S[4]/J_b^2$. We will show in what follows that actually $V(r_1)=\a6\setminus\cU$.

Consider $ r_2.$ There exists $b\in V(r_2)\cap\cU$: since $\cU$ is open this property holds for the general $b\in V(r_2)$. The conditions that force the initial ideal to change obviously contain $ r_2 $. They are:
\begin{itemize}
\item  $ r_{2,1} = (  b_1, b_3, b_4, b_5 );$
\item  $ r_{2,2} = (  b_0, b_1, b_3, b_4, b_5 );$
\item  $ r_{2,3} = (  -b_1^2+b_0b_2, b_3, b_4, b_5 );$
\item  $ r_{2,4} = (  b_1, b_2, b_3, b_4, b_5 ).$
\end{itemize}
Notice that $ r_{2,2}, r_{2,4} \supset r_{2,1} $. On the one hand, if $b\in V(r_{2,2})\cup V(r_{2,4})\cup V(r_{2,1})=V(r_{2,1})$ using Singular, we check that $N_{S[4]/J_b} = 49$, whence $b\not\in\cU$. On the other hand, if $b\in V(r_{2,3})\setminus V(r_{2,1})$, then $b\in\cU$. We notice that $r_{2,1} \supset r_1$.

Further, we consider $ r_8.$ We first notice that there exists $b\in V(r_8)\cap \cU$. There are many conditions that force the initial ideal to change, but many of them have been already examined above (e.g. $r_{2,2}$). An easy and careful check of such conditions yields that the only totally new ones are
\begin{itemize}
\item  $ r_{8,1} = (  b_0, b_1, b_3, b_4 );$
\item  $ r_{8,2} = (  b_0, b_1, b_2, b_3, b_4 );$
\item  $ r_{8,3} = (  -b_4^2+b_2b_5, b_3, b_1, b_0 ).$
\end{itemize}
Notice that $V(r_{8,1}),V( r_{8,2})\subseteq V(r_1) \subseteq\a6\setminus\cU$. If $b\in V(r_{8,3})\setminus V(r_1)$, then $b\in\cU$. We conclude that there are  no level $ 3 $ conditions in this case.

Also for $V(r_{13})$, there exists $b\in V(r_{13})\cap \cU$. The conditions that force the initial ideal to change are:
\begin{itemize}
\item  $ r_{13,1} = (  b_1, b_2, b_3, b_4 );$
\item  $ r_{13,2} = (  -b_3^2+b_0b_5, b_4, b_2, b_1 ).$
\end{itemize}
The first one defines a subvariety of $ V(r_1)$ and so we do not study further it. If $b\in V(r_{13,2})\setminus V(r_1)$, then $b\in\cU$. A similar argument holds for $V(r_{17})\setminus V(r_1)$.

Now, we start studying codimension $ 2 $ ideals that appear in the initial list. The first we consider is $ r_5.$ Again there is $b\in V(r_5)\cap \cU$, and the only new condition that forces the initial ideal to change is
\begin{itemize}
\item  $ r_{5,1} = (  -b_4^2+b_2b_5, b_3, b_1 ).$
\end{itemize}
$b\in\cU$ if, and only if, $b\not\in V(r_1)$. Thus there is no level $ 3 $ condition to analyze. A similar argument holds for $ r_6 $ and $ r_7$. In this case the only new conditions are:
\begin{itemize}
\item  $ r_{6,1} = (  -b_3^2+b_0b_5, b_4, b_1 );$
\item  $ r_{7,1} = (  -b_1^2+b_0b_2, b_4, b_3 ).$
\end{itemize}
$b\in\cU$ if, and only if, either $b\in V(r_{6,1})\setminus V(r_1)$ or $b\in V(r_{7,1})\setminus V(r_1)$.

Next, we examine $V(r_{14})$. Again there is $b\in V(r_{14})\cap\cU$. The conditions that force $\init(r_{14} + J^2_b)$ to change are:
\begin{itemize}
\item  $ r_{14,1} = (  b_2, b_1, b_0 );$
\item  $ r_{14,2} = (  b_3, b_2, b_1, b_0 ).$
\end{itemize}
Notice that $V( r_{14,2})\subseteq V(r_8)$. We already studied $V(r_8)$ above,
and so we forget it because it does not give new insight. Again there is
$b\in V(r_{14,1})\cap\cU$, but a new condition appears that forces
$\init(r_{14,1} + J^2_b) $ to change, i.e. a level $ 3 $ condition. It is:
\begin{itemize}
\item  $ r_{14,1,1} = (  b_4, b_2, b_1, b_0 ).$
\end{itemize}
Trivially $r_{14,1,1}\subseteq V(r_{13}) $ that we studied above.

When studying $ r_{15} $ no new condition shows up, and $V(r_{15})\cap\cU\ne\emptyset$.

The last codimension $ 2 $ condition is $ r_{16}.$ Again $V(r_{16})\cap\cU\ne\emptyset$, and the special condition is:
\begin{itemize}
\item  $ r_{16,1} = (  b_5, b_3, b_0 ).$
\end{itemize}
Studying $ r_{16,1},$ the following two new conditions appear
\begin{itemize}
\item  $ r_{16,1,1} = (  b_5, b_3, b_1, b_0 );$
\item  $ r_{16,1,2} = (  b_5, b_4, b_3, b_0 ).$
\end{itemize}
The first one defines a subvariety of $ V(r_8) $ while the second one defines a subvariety of $ V( r_2) $ and so no one of them has to be studied. Moreover the points outside $V(r_{16,1,1})\cup V(r_{16,1,2})$ are in $\cU$.

Now, we consider codimension $ 1 $ ideals.

The first one we consider is $ r_3.$ For general $b\in V(r_3)$, then $b\in\cU$, and the new conditions forcing the initial ideal to change are:
\begin{itemize}
\item  $ r_{3,1} = (  -b_1^2+b_0b_2, b_3 );$
\item  $ r_{3,2} = (  b_0, b_3 );$
\item  $ r_{3,3} = (  -b_0b_4^2-b_1^2b_5+b_0b_2b_5, b_3 ).$
\end{itemize}
There exists $b\in V(r_{3,2})\cap\cU$, and the only new condition is:
\begin{itemize}
\item  $ r_{3,2,1} = (  b_4, b_3, b_0 ).$
\end{itemize}
Such a case has been already examined above because $ V(r_{3,2,1})\subseteq V(r_7)$.

When we take a general $b\in V(r_{3,1}),$ then again $b\in\cU$. The only new condition to examine is:
\begin{itemize}
\item  $ r_{3,1,1} = (  b_3, b_2, b_1 ). $
\end{itemize}
We can argue as above because $ V(r_{3,1,1})\subseteq V(r_5)$.

Finally, $b\in\cU$ is unobstructed for general $b\in V(r_{3,3})$ and no new condition shows up. Hence the analysis of $ r_3 $ is over.

We find $b\in V(r_4)\cap\cU$, thus the general $b\in V(r_4)$ is in $\cU$. The conditions that force the $\init(r_4 + J^2) $ to change are:
\begin{itemize}
\item  $ r_{4,1} = (  b_2, b_1 );$
\item  $ r_{4,2} = (  b_1, b_0 );$
\item  $ r_{4,3} = (  -b_2b_3^2-b_0b_4^2+b_0b_2b_5, b_1 ).$
\end{itemize}
If $b\in V(r_{4,1})\cup V(r_{4,3})$ is general, then $b\in\cU$. Since no new condition appears it follows that the same holds for each $b$. Similarly $b\in V(r_{4,2})\cap\cU$, but the following two new conditions show up:
\begin{itemize}
\item  $ r_{4,2,1} = (  b_4, b_1, b_0 );$
\item  $ r_{4,2,2} = (  b_2^3-b_4^3+b_2b_4b_5, b_3, b_1, b_0 ).$
\end{itemize}
The first ideal defines a subvariety of $ V(r_6) $ while the second one defines a subvariety of $ V(r_8) $ and so no one of them has to be studied.

The analysis of $ r_9 $ is more difficult than the previous ones. We still have $V(r_9)\cap\cU\ne\emptyset$, and the only new condition that appears is:
\begin{itemize}
\item  $ r_{9,1} = (  b_1b_2^2b_3^2-2b_1^2b_2b_3b_4+b_1^3b_4^2-b_3^3b_4^2+2b_1b_3^2b_4b_5-b_1^2b_3b_5^2, -b_1^2+b_0b_2 ).$
\end{itemize}
Again, we can find $b\in V(r_{9,1})\cap\cU$, but a new condition appears:
\begin{itemize}
\item  $ r_{9,1,1} = (  b_2b_3+b_1b_4, b_1b_3+b_0b_4, b_1^2-b_0b_2, 4b_2^3b_4+b_4^4+2b_2b_4^2b_5+b_2^2b_5^2, 4b_1b_2^2b_4-b_3b_4^3+2b_1b_4^2b_5+b_1b_2b_5^2, 4b_0b_2^2b_4+b_3^2b_4^2+2b_0b_4^2b_5+b_0b_2b_5^2, 4b_0b_1b_2b_4-b_3^3b_4-2b_0b_3b_4b_5+b_0b_1b_5^2, b_3^4+4b_0^2b_2b_4+2b_0b_3^2b_5+b_0^2b_5^2 ).$
\end{itemize}
Though $V(r_{9,1,1})\cap\cU\ne\emptyset$. Several non--generic new conditions show up:
\begin{itemize}
\item  $ r_{9,1,1,1} = (  4b_2^3b_4+b_4^4+2b_2b_4^2b_5+b_2^2b_5^2,
b_3, b_1, b_0 );$
\item  $ r_{9,1,1,2} = (  b_3+b_4, b_1-b_2, b_0-b_2,
4b_2^3b_4+b_4^4+2b_2b_4^2b_5+b_2^2b_5^2 );$
\item  $ r_{9,1,1,3} = (  b_4, b_2, b_1, b_3^2+b_0b_5 );$
\item  $ r_{9,1,1,4} = (  b_0+b_1+b_2, b_3^2-b_3b_4+b_4^2, b_2b_3+b_1b_4, b_1b_3-b_1b_4-b_2b_4, b_1^2+b_1b_2+b_2^2, 4b_2^3b_4+b_4^4+2b_2b_4^2b_5+b_2^2b_5^2, 4b_1b_2^2b_4-b_3b_4^3+2b_1b_4^2b_5+b_1b_2b_5^2 );$
\item  $ r_{9,1,1,5} = (  b_3, b_1, b_0, 3b_4^2-b_2b_5, 9b_2b_4+4b_5^2, 3b_2^2+4b_4b_5 );$
\item  $ r_{9,1,1,6} = (  b_4+b_5, b_3, b_2-b_5, b_1, b_0 );$
\item  $ r_{9,1,1,7} = (  b_4^2-b_4b_5+b_5^2,
b_3, b_2-b_4+b_5, b_1, b_0 ).$
\end{itemize}
The ideals $ r_{9,1,1,1}, r_{9,1,1,5}, r_{9,1,1,6} $ and $ r_{9,1,1,7} $ define subvarieties of $ V(r_8),$ while $ r_{9,1,1,3} $ defines a subvariety of $ V(r_{13}) $ and so we do not study them.

We have $V(r_{9,1,1,2})\cap\cU\not=\emptyset$, and the new non--general conditions are:
\begin{itemize}
\item  $ r_{9,1,1,2,1} = (  b_5, b_4, b_3, b_1-b_2, b_0-b_2 );$
\item  $ r_{9,1,1,2,2} = (  b_4^2-b_4b_5+b_5^2,
b_2-b_4+b_5, b_3+b_4, b_1-b_2, b_0-b_2 );$
\item  $ r_{9,1,1,2,3} = (  b_4+b_5, b_2+b_4, b_3+b_4, b_1-b_2, b_0-b_2 ).$
\end{itemize}
The first ideal defines a subvariety of $ V( r_2) $ and so we do not study it. The general $b\in V(r_{9,1,1,2,2} )\cup V(r_{9,1,1,2,3})$ is in $\cU$, and in both cases, the non--general condition is $ b_0 = \dots = b_5 = 0,$ that defines a point in $ V(r_1).$

We have $V(r_{9,1,1,4})\cap\cU\ne\emptyset$, and the non--general conditions are
\begin{itemize}
\item  $ r_{9,1,1,4,1} = (  b_4^2-b_4b_5+b_5^2,
b_3-b_4+b_5, b_2-b_3, b_1-b_5, b_0+b_1+b_2 );$
\item  $ r_{9,1,1,4,2} = (  b_5, b_4, b_3, b_1^2+b_1b_2+b_2^2, b_0+b_1+b_2 );$
\item  $ r_{9,1,1,4,3} = (  b_4^2-b_4b_5+b_5^2,
b_3-b_5, b_2-b_4+b_5, b_1+b_4, b_0+b_1+b_2 );$
\item  $ r_{9,1,1,4,4} = (  b_4+b_5, b_3^2+b_3b_5+b_5^2,
b_2+b_4, b_1-b_3, b_0+b_1+b_2 ).$
\end{itemize}
$ V(r_{9,1,1,4,2})\subseteq V( r_2) $ and so we do
not consider it. The general  $b\in V(r_{9,1,1,4,1})\cup V( r_{9,1,1,4,3})\cup V(  r_{9,1,1,4,4})$
is in $\cU$, and the only non--general condition in all the
three cases is $ b_0 = \dots = b_5 = 0,$ that was considered
earlier. So, the analysis of $ r_9 $ is over, too.

We can find $b\in V( r_{10})\cap\cU$ and the only new non--general condition that appears
from the study of $\init( r_{10} + J^2 )$ is:
\begin{itemize}
\item  $ r_{10,1} = (  b_1b_3-b_0b_4,
b_0b_2^2b_3+b_1^3b_4-2b_0b_1b_2b_4-b_3^3b_4+2b_0b_3b_4b_5-b_0b_1b_5^2
).$
\end{itemize}
Again, $V(r_{10,1})\cap\cU\ne\emptyset$, and no new non--general condition shows up.
Hence, the analysis of $ r_{10} $ is complete.

$V(r_{12})\cap\cU\ne\emptyset$ and no new
non--general condition appears. We infer $V(r_{12})\setminus V(r_1)\subseteq\cU$.

So, we can consider the last and more difficult case $r_{11}$. As usual $V(r_{11})\cap\cU\ne\emptyset$,
and the new non--general conditions are:
\begin{itemize}
\item  $ r_{11,1} = (
b_2b_3^2-2b_1b_3b_4+b_0b_4^2+b_1^2b_5-b_0b_2b_5,
b_1^3-b_0b_1b_2-b_3^3+b_0b_3b_5 );$
\item  $
r_{11,2} = (  b_1b_3+b_0b_4,
b_0b_2^2b_3-b_1^3b_4-2b_0b_1b_2b_4+b_3^3b_4+2b_0b_3b_4b_5-b_0b_1b_5^2
).$
\end{itemize}
$ r_{11,1} $ contains $ r_{12} $ and so we can skip its study.
However, there is $b\in V(
r_{11,2})\cap\cU$. The new non--general conditions that
show up from the study of $\init(r_{11,2} + J^2_b)$
are:
\begin{itemize}
\item  $ r_{11,2,1} = (  b_4^2-b_2b_5,
b_3b_4-b_1b_5, b_3^2+b_0b_5, b_2b_3-b_1b_4, b_1b_3+b_0b_4,
b_1^2+b_0b_2 );$
\item  $ r_{11,2,2} = (
b_5, b_4, b_3, b_1^2+b_0b_2 );$
\item  $
r_{11,2,3} = (  b_4, b_2, b_1, b_3^2+b_0b_5 );$
\item  $ r_{11,2,4} = (  b_3, b_1, b_0,
4b_2^3b_4-27b_4^4+18b_2b_4^2b_5+b_2^2b_5^2+4b_4b_5^3 );$
\item  $ r_{11,2,5} = (  b_3, b_1, b_0,
4b_2^3b_4-15b_4^4+4b_2b_4^2b_5+3b_2^2b_5^2+4b_4b_5^3 );$
\item  $ r_{11,2,6} = (  9b_4^2-3b_4b_5+b_5^2, b_3,
b_2-3b_4+b_5, b_1, b_0 );$
\item  $ r_{11,2,7} =
(  3b_4+b_5, b_3, b_2+3b_4, b_1, b_0 );$
\item  $ r_{11,2,8} = (  b_5, b_3, b_2, b_1, b_0
);$
\item  $ r_{11,2,9} = (
b_4^2+b_4b_5+b_5^2, b_3, b_2+b_4+b_5, b_0, b_1 );$
\item  $ r_{11,2,10} = (  b_4-b_5, b_3, b_2-b_4,
b_0, b_1 );$
\item  $ r_{11,2,11} = (
225b_4^3+17b_5^3, b_3, -105b_4^2+17b_2b_5, 15b_2b_4+7b_5^2,
17b_2^2+49b_4b_5, b_0, b_1 );$
\item  $
r_{11,2,12} = (  b_1b_3+b_0b_4,
3b_1^2b_4+b_0b_2b_4-b_3^2b_5+b_0b_5^2,
b_2b_3^2+3b_0b_4^2+b_1^2b_5-b_0b_2b_5,
b_2^2b_3-3b_1b_2b_4+3b_3b_4b_5-b_1b_5^2,
b_1^2b_2-b_0b_2^2-3b_3^2b_4-b_0b_4b_5,
b_1^3-b_0b_1b_2-b_3^3+b_0b_3b_5,
3b_2b_3b_4^2-9b_1b_4^3+4b_1b_2b_4b_5+b_3b_4b_5^2+b_1b_5^3,
4b_2^3b_4-27b_4^4+18b_2b_4^2b_5+b_2^2b_5^2+4b_4b_5^3,
4b_1b_2^2b_4-9b_3b_4^3+b_2b_3b_4b_5+3b_1b_4^2b_5+b_1b_2b_5^2
);$
\item  $ r_{11,2,13} = (  b_1b_3+b_0b_4,
b_0b_2^2b_3-b_1^3b_4-2b_0b_1b_2b_4+b_3^3b_4+2b_0b_3b_4b_5-b_0b_1b_5^2,
4b_0b_1^3b_4+12b_0^2b_1b_2b_4+b_2b_3^2b_4^2+3b_0b_4^4-b_2^2b_3^2b_5-12b_0^2b_3b_4b_5+b_1^2b_4^2b_5-4b_0b_2b_4^2b_5+4b_0^2b_1b_5^2-b_1^2b_2b_5^2+b_0b_2^2b_5^2,
b_2^2b_3^4+16b_0^3b_1b_2b_4+12b_0^2b_3^3b_4+6b_0b_2b_3^2b_4^2+9b_0^2b_4^4-8b_0^3b_3b_4b_5+2b_3^4b_4b_5+8b_0b_1^2b_4^2b_5+4b_0^3b_1b_5^2+b_1^4b_5^2-2b_0b_1^2b_2b_5^2+b_0^2b_2^2b_5^2+4b_0b_3^2b_4b_5^2+2b_0^2b_4b_5^3,
b_2^3b_3^3+12b_0^2b_2b_3^2b_4+4b_0^3b_4^3+2b_1^3b_4^3+3b_0b_1b_2b_4^3-5b_3^3b_4^3-4b_0^2b_1^2b_4b_5+4b_0^3b_2b_4b_5+b_2b_3^3b_4b_5-11b_0b_3b_4^3b_5+2b_0b_2b_3b_4b_5^2+7b_0b_1b_4^2b_5^2+b_1^3b_5^3-b_0b_1b_2b_5^3,
b_2^4b_3^2-16b_0^2b_1b_2b_4^2+b_1^2b_2^2b_4^2+2b_0b_2^3b_4^2-12b_0b_3^3b_4^2-6b_2b_3^2b_4^3-9b_0b_4^5+4b_2^2b_3^2b_4b_5+8b_0^2b_3b_4^2b_5-6b_1^2b_4^3b_5+4b_0b_2b_4^3b_5-4b_0^2b_1b_4b_5^2+4b_1^2b_2b_4b_5^2+b_0b_2^2b_4b_5^2+b_3^2b_4^2b_5^2+2b_0b_4^2b_5^3+b_1^2b_5^4,
4b_0^3b_1^3+b_1^6-4b_0^4b_1b_2+b_0b_1^4b_2-b_0^2b_1^2b_2^2-b_0^3b_2^3-12b_0^3b_3^3+b_3^6-4b_0^2b_2b_3^2b_4-4b_0^3b_4^3-4b_0^4b_3b_5+b_0b_3^4b_5-4b_0^2b_1^2b_4b_5-4b_0^3b_2b_4b_5-b_0^2b_3^2b_5^2-b_0^3b_5^3,
b_2^3b_3^2b_4^2+12b_0^2b_2b_3b_4^3-4b_0^2b_1b_4^4+b_1^2b_2b_4^4+2b_0b_2^2b_4^4-3b_3^2b_4^5-16b_0^2b_1b_2b_4^2b_5-10b_0b_4^5b_5+b_2^2b_3^2b_4b_5^2+12b_0^2b_3b_4^2b_5^2-6b_1^2b_4^3b_5^2+6b_0b_2b_4^3b_5^2-4b_0^2b_1b_4b_5^3+4b_1^2b_2b_4b_5^3+b_3^2b_4^2b_5^3+2b_0b_4^2b_5^4+b_1^2b_5^5,
4b_0b_1^2b_2^2b_4^2-4b_0^2b_2^3b_4^2-4b_0b_2b_3^2b_4^3-b_2^2b_3b_4^4+3b_1b_2b_4^5+2b_2^3b_3b_4^2b_5+4b_0b_1^2b_4^3b_5+16b_0^2b_2b_4^3b_5-6b_1b_2^2b_4^3b_5-3b_3b_4^5b_5-b_2^4b_3b_5^2+3b_1b_2^3b_4b_5^2+12b_0b_3^2b_4^2b_5^2+6b_2b_3b_4^3b_5^2+b_1b_4^4b_5^2-3b_2^2b_3b_4b_5^3+4b_0^2b_4^2b_5^3-2b_1b_2b_4^2b_5^3+b_1b_2^2b_5^4,
16b_0^2b_1b_2^2b_4^2+b_2^2b_3^2b_4^3-4b_0^2b_3b_4^4-3b_1^2b_4^5+2b_0b_2b_4^5-b_2^3b_3^2b_4b_5-12b_0^2b_2b_3b_4^2b_5-4b_0^2b_1b_4^3b_5+7b_1^2b_2b_4^3b_5-2b_0b_2^2b_4^3b_5+b_3^2b_4^4b_5+4b_0^2b_1b_2b_4b_5^2-4b_1^2b_2^2b_4b_5^2-b_2b_3^2b_4^2b_5^2+2b_0b_4^4b_5^2+b_1^2b_4^2b_5^3-2b_0b_2b_4^2b_5^3-b_1^2b_2b_5^4,
48b_0^4b_2b_3b_4-4b_0b_2b_3^4b_4-16b_0^4b_1b_4^2+48b_0^3b_2^2b_4^2-3b_2^2b_3^3b_4^2-6b_0b_2b_3b_4^4-9b_0b_1b_4^5+48b_0^3b_3^3b_5-4b_3^6b_5+8b_0^2b_2b_3^2b_4b_5-32b_0^3b_4^3b_5+24b_0b_1b_2b_4^3b_5+6b_3^3b_4^3b_5+16b_0^4b_3b_5^2-4b_0b_3^4b_5^2+16b_0^2b_1^2b_4b_5^2+12b_0^3b_2b_4b_5^2-12b_0b_1b_2^2b_4b_5^2+15b_0b_3b_4^3b_5^2+4b_0^2b_3^2b_5^3-12b_0b_1b_4^2b_5^3+4b_0^3b_5^4-3b_1^3b_5^4,
12b_0^3b_2b_3^3-b_2b_3^6+4b_0^4b_3b_4^2-3b_0b_3^4b_4^2+4b_0^4b_2b_3b_5-b_0b_2b_3^4b_5+4b_0^4b_1b_4b_5-7b_0^2b_3^2b_4^2b_5+b_0^2b_2b_3^2b_5^2-5b_0^3b_4^2b_5^2-b_0^2b_1^2b_5^3+b_0^3b_2b_5^3,
8b_1^4b_4^4+28b_0b_1^2b_2b_4^4-4b_0^2b_2^2b_4^4-4b_0b_3^2b_4^5-3b_2b_3b_4^6+9b_1b_4^7+16b_0^2b_2^3b_4^2b_5+16b_0b_2b_3^2b_4^3b_5+6b_2^2b_3b_4^4b_5+32b_0^2b_4^5b_5-22b_1b_2b_4^5b_5-3b_2^3b_3b_4^2b_5^2+4b_0b_1^2b_4^3b_5^2-4b_0^2b_2b_4^3b_5^2+17b_1b_2^2b_4^3b_5^2-b_3b_4^5b_5^2+4b_0^2b_2^2b_4b_5^3-4b_1b_2^3b_4b_5^3+4b_0b_3^2b_4^2b_5^3+2b_2b_3b_4^3b_5^3-b_1b_4^4b_5^3-b_2^2b_3b_4b_5^4-4b_0^2b_4^2b_5^4+2b_1b_2b_4^2b_5^4-b_1b_2^2b_5^5,
16b_0^2b_2^4b_4^2-16b_3^4b_4^4-28b_0b_1^2b_4^5-64b_0^2b_2b_4^5-4b_1b_2^2b_4^5+9b_3b_4^7-4b_0b_1^2b_2b_4^3b_5-4b_0^2b_2^2b_4^3b_5+8b_1b_2^3b_4^3b_5-68b_0b_3^2b_4^4b_5-19b_2b_3b_4^5b_5-3b_1b_4^6b_5+4b_0^2b_2^3b_4b_5^2-4b_1b_2^4b_4b_5^2+4b_0b_2b_3^2b_4^2b_5^2+11b_2^2b_3b_4^3b_5^2-28b_0^2b_4^4b_5^2+5b_1b_2b_4^4b_5^2-b_2^3b_3b_4b_5^3-4b_0^2b_2b_4^2b_5^3-b_1b_2^2b_4^2b_5^3-b_1b_2^3b_5^4,
768b_0^4b_1b_2b_4^2+128b_0^3b_2^3b_4^2+16b_1^3b_2^3b_4^2+32b_0b_1b_2^4b_4^2+576b_0^3b_3^3b_4^2-48b_3^6b_4^2+320b_0^2b_2b_3^2b_4^3+416b_0^3b_4^5-68b_1^3b_4^5-192b_0b_1b_2b_4^5+4b_2^3b_4^5+8b_3^3b_4^5-27b_4^8-384b_0^4b_3b_4^2b_5-1088b_0^3b_2b_4^3b_5+16b_1^3b_2b_4^3b_5+104b_0b_1b_2^2b_4^3b_5-8b_2^4b_4^3b_5+112b_2b_3^3b_4^3b_5+392b_0b_3b_4^5b_5+72b_2b_4^6b_5+192b_0^4b_1b_4b_5^2-160b_0^3b_2^2b_4b_5^2+4b_1^3b_2^2b_4b_5^2-8b_0b_1b_2^3b_4b_5^2+4b_2^5b_4b_5^2-88b_2^2b_3^3b_4b_5^2-1120b_0^2b_3^2b_4^2b_5^2-408b_0b_2b_3b_4^3b_5^2+20b_0b_1b_4^4b_5^2-62b_2^2b_4^4b_5^2-192b_0^3b_4^2b_5^3+176b_1^3b_4^2b_5^3+
160b_0b_1b_2b_4^2b_5^3+16b_2^3b_4^2b_5^3-128b_3^3b_4^2b_5^3+4b_4^5b_5^3-48b_0^2b_1^2b_5^4-48b_1^3b_2b_5^4+
44b_0b_1b_2^2b_5^4+b_2^4b_5^4-272b_0b_3b_4^2b_5^4-8b_2b_4^3b_5^4+128b_0b_1b_4b_5^5+4b_2^2b_4b_5^5
).$
\end{itemize}
We have $V(r_{11,2,j})\subseteq V(r_8) $, $
j=4, \dots, 11$, $V( r_{11,2,1})\subseteq V(r_{15})$, $V(r_{11,2,2})\subseteq V( r_2),$ $
V(r_{11,2,3})\subseteq V(r_{13}) $ and finally $
V(r_{11,2,12})\subseteq V(r_{12}).$ Hence, the
only ideal to be studied is $ r_{11,2,13}.$

Again there is $b\in V( r_{11,2,13}) \cap\cU$, but many new non--general conditions appear. They are:
\begin{itemize}
\item  $ r_{11,2,13,1} = (  3b_4+b_5, b_3, b_2-b_5,
b_1, b_0 );$
\item  $ r_{11,2,13,2} = (
4b_0^3b_1+b_1^4+2b_0b_1^2b_2+b_0^2b_2^2, b_5, b_4, b_3 );$
\item  $ r_{11,2,13,3} = (
12b_0^3b_3^3-b_3^6+4b_0^4b_3b_5-b_0b_3^4b_5+b_0^2b_3^2b_5^2+b_0^3b_5^3,
b_4, b_2, b_1 );$
\item  $ r_{11,2,13,4} = (
b_4-b_5, b_2-b_5, b_1-b_3, b_3^2+b_0b_5 );$
\item
$ r_{11,2,13,5} = (  b_2+b_4+b_5, b_4^2+b_4b_5+b_5^2,
b_3b_4-b_1b_5, b_1b_4+b_1b_5+b_3b_5, b_3^2+b_0b_5, b_1b_3+b_0b_4,
b_1^2-b_0b_4-b_0b_5 );$
\item  $ r_{11,2,13,6} =
(  b_4^2-b_4b_5+b_5^2, b_3, b_2-b_4+b_5, b_1, b_0 ),$
\item  $ r_{11,2,13,7} = (  b_4+b_5, b_3, b_2+b_4,
b_1, b_0 );$
\item  $ r_{11,2,13,8} = (
b_3+b_4, b_1-b_2, b_0-b_2, 4b_2^3b_4+b_4^4+2b_2b_4^2b_5+b_2^2b_5^2
);$
\item  $ r_{11,2,13,9} = (  b_0+b_1+b_2,
b_3^2-b_3b_4+b_4^2, b_2b_3+b_1b_4, b_1b_3-b_1b_4-b_2b_4,
b_1^2+b_1b_2+b_2^2, 4b_2^3b_4+b_4^4+2b_2b_4^2b_5+b_2^2b_5^2,
4b_1b_2^2b_4-b_3b_4^3+2b_1b_4^2b_5+b_1b_2b_5^2 );$
\item  $ r_{11,2,13,10} = (  b_5, b_4, b_3,
b_1+b_2, b_0-b_2 );$
\item  $ r_{11,2,13,11} =
(  b_5, b_4, b_3, b_1^2-b_1b_2+b_2^2, b_0-b_1+b_2 ),$
\item  $ r_{11,2,13,12} = (
4b_3^3+3b_4^3-2b_4^2b_5-b_4b_5^2, b_2-b_5, b_1-b_3, b_3^2+b_0b_4,
4b_0b_3-3b_4^2+2b_4b_5+b_5^2 );$
\item  $
r_{11,2,13,13} = (  b_2b_3+b_1b_5+b_3b_5, b_1b_3+b_0b_4,
b_2^2+b_2b_5+b_5^2, b_1b_2-b_3b_5, b_1^2+b_3^2-b_0b_4,
4b_0b_1+4b_0b_3+2b_2b_4+3b_4^2-b_2b_5+2b_4b_5,
4b_0^2b_3-3b_3^2b_4+3b_0b_4^2-2b_3^2b_5+b_0b_5^2 ).$
\end{itemize}
The only ideals to be studied are $ r_{11,2,13,12} $ and $
r_{11,2,13,13} $ because $V(r_{11,2,13,j})$ is contained in either $V( r_2)$ (in the cases $j=2,10,11$), or $V(r_8)$ (in the cases $j=1,6,7$), or $ V(r_{12}) $ (in the cases $j=4,5$), or $ V(r_9) $ (in the cases $j=8,9$).

For general $b\in V(r_{11,2,13,12} )\cup V(r_{11,2,13,13})$, then $b\in\cU$. The new non--general
conditions inside $r_{11,2,13,12}$ are:
\begin{itemize}
\item  $ r_{11,2,13,12,1} = (  b_4+b_5, b_3-b_5,
b_0-b_5, b_2-b_5, b_1-b_3 );$
\item  $
r_{11,2,13,12,2} = (  b_4+b_5, b_3^2+b_3b_5+b_5^2,
b_0+b_3+b_5, b_2-b_5, b_1-b_3 ),$
\end{itemize}
but both ideals define subvarieties of $ V(r_9) $ and so we do not
study them.

The new non--general conditions
 inside $r_{11,2,13,13}$ are:
\begin{itemize}
\item  $ r_{11,2,13,13,1} = (  b_4^2-b_4b_5+b_5^2,
b_3+b_4, b_2-b_4+b_5, b_1-b_4+b_5, b_0-b_4+b_5 );$
\item  $ r_{11,2,13,13,2} = (  b_4^2-b_4b_5+b_5^2,
b_3-b_4+b_5, b_2-b_4+b_5, b_1-b_5, b_0+b_4 );$
\item  $ r_{11,2,13,13,3} = (  b_4^2-b_4b_5+b_5^2,
b_3-b_5, b_2-b_4+b_5, b_1+b_4, b_0-b_5 ) $
\end{itemize}
but they all define subvarieties of $ V(r_9) $ and so they are not
to be studied, and the analysis of $ r_{11} $ is over.

Our analysis in now complete when $H=y_2^3+y_3^3+y_4^3$.

Let  $ H = y_2^3 + y_3^2y_4 $. $\Ann(F^{H,b}) $ is generated by
\begin{gather*}
12x_4^2-b_5x_1^3,\quad 12x_4x_2-b_3x_1^3,\quad 12x_3x_2-b_1x_1^3,\quad
3x_2x_1-b_0p_1-3b_1p_2-3b_3p_3,\\ 3x_3x_1-b_1p_1-3b_2p_2-3b_4p_3,\quad
 3x_4x_1-b_3p_1-3b_4p_2-3b_5p_3,\quad
x_2x_1^2, \\ 12x_2^2x_1-b_0x_1^4,\quad 4x_2^3-x_1^4, \quad x_3x_1^2,\quad
12x_3x_2x_1-b_1x_1^4, \quad x_3x_2^2, \quad
 12x_3^2x_1-b_2x_1^4, \\ x_3^2x_2,\quad
x_3^3, \quad x_4x_1^2, \quad 12x_4x_2x_1-b_3x_1^4,\quad  x_4x_2^2,\quad
12x_4x_3x_1-b_4x_1^4, \quad x_4x_3x_2, \\ 12x_4x_3^2-x_1^4,\quad
12x_4^2x_1-b_5x_1^4, \quad x_4^2x_2, \quad x_4^2x_3, \quad x_4^3,
\end{gather*}
where
$$
p_1 = x_2^2-b_0x_1^3/12,\quad p_2 = x_4x_3-b_4x_1^3/12,\quad p_3 = x_3^2-b_2x_1^3/12.
$$
Let $ H = y_2^2y_3^2y_4$. $\Ann(F^{H,b}) $ is generated by
\begin{gather*}
12x_4^2-b_5x_1^3,\quad 12x_4x_2-b_3x_1^3,\quad
x_2x_1-b_0p_1-b_1p_2-b_3p_3,\quad x_2^2-b_0x_1^3/12-p_2,\\
x_3x_1-b_1p_1-b_2p_2-b_4p_3,\quad x_4x_1-b_3p_1-b_4p_2-b_5p_3,\quad
x_2x_1^2, \quad
 12x_2^2x_1-b_0x_1^4, \\
 x_2^3, \quad x_3x_1^2, \quad
12x_3x_2x_1-b_1x_1^4, \quad 12x_3x_2^2-x_1^4, \quad
 12x_3^2x_1-b_2x_1^4, \quad
x_3^2x_2, \quad x_3^3, \\ x_4x_1^2, \quad 12x_4x_2x_1-b_3x_1^4, \quad x_4x_2^2, \quad
12x_4x_3x_1-b_4x_1^4, \quad x_4x_3x_2,\quad 12x_4x_3^2-x_1^4, \\
12x_4^2x_1-b_5x_1^4, \quad x_4^2x_2, \quad x_4^2x_3, \quad x_4^3,
\end{gather*}
where
$$
p_1 = x_3x_2-b_1x_1^3/12, \quad p_2 = x_4x_3-b_4x_1^3/12, \quad p_3 = x_3^2-b_2x_1^3/12.
$$
Let $ H = y_3^2y_4-y_3y_4^2$.
$ b_0$, $b_1$, $b_3 $ cannot vanish simultaneously.  $\Ann(F^{H,b}) $ is generated by
\begin{gather*}
12x_2^2-b_0x_1^3,\quad 12x_3x_2-b_1x_1^3,\quad 12x_4x_2-b_3x_1^3,\\
12x_3^2+12x_4x_3+12x_4^2-(b_2+b_4+b_5)x_1^3,\quad
b_1p_1-b_0p_2,\quad
b_3p_1-b_0p_3,\quad b_3p_2-b_1p_3,\\
x_4^3,\quad  12x_4^2x_3+x_1^4, x_4^2x_2,\quad
12x_4^2x_1-b_5x_1^4,\quad
12x_4x_3^2-x_1^4,\quad  x_4x_3x_2,\\
12x_4x_3x_1-b_4x_1^4,\quad  x_4x_2^2,\quad  12x_4x_2x_1-b_3x_1^4,\quad  x_3^3,\quad
x_3^2x_2,\quad  12x_3^2x_1-b_2x_1^4,\\
x_3x_2^2,\quad  12x_3x_2x_1-b_1x_1^4,\quad
x_2^3,\quad  12x_2^2x_1-b_0x_1^4,\quad  x_4x_1^2,\quad  x_3x_1^2,\quad  x_2x_1^2,
\end{gather*}
where
\begin{gather*}
p_1 = x_2x_1 +b_1(x_4^2-b_5x_1^3/12)-b_3(x_3^2-b_2x_1^3/12),\\
p_2 = x_3x_1 +b_2(x_4^2-b_5x_1^3/12)-b_4(x_3^2-b_2x_1^3/12),\\
p_3 = x_4x_1 +b_4(x_4^2-b_5x_1^3/12)-b_5(x_3^2-b_2x_1^3/12).
\end{gather*}
Let $ H = y_3y_4^2 $.
$ b_0$, $b_1$, $b_3 $ cannot vanish simultaneously and $\Ann(F^{H,b}) $ is generated by
\begin{gather*}
12x_2^2-b_0x_1^3,\quad 12x_3x_2-b_1x_1^3,\quad 12x_4x_2-b_3x_1^3,\quad
12x_3^2-b_2x_1^3,\quad
b_1p_1-b_0p_2,\\  b_3p_1-b_0p_3,\quad  b_3p_2-b_1p_3,\quad x_4^3,\quad
12x_4^2x_3-x_1^4,\quad x_4^2x_2,\quad x_4x_3^2,\quad x_4x_3x_2,\quad x_4x_2^2,\\
x_3^3,\quad x_3^2x_2,\quad x_3x_2^2,\quad x_2^3,\quad
12x_4^2x_1-b_5x_1^4,\quad
12x_4x_3x_1-b_4x_1^4,\quad 12x_4x_2x_1-b_3x_1^4,\\ 12x_3^2x_1-b_2x_1^4,\quad
12x_3x_2x_1-b_1x_1^4,\quad
12x_2^2x_1-b_0x_1^4,\quad x_4x_1^2,\quad x_3x_1^2,\quad
x_2x_1^2,
\end{gather*}
where
\begin{gather*}
p_1 = x_2x_1 - b_1(x_4^2 - b_5 x_1^3 / 12) - b_3 (x_4x_3 -
b_4 x_1^3/12),\\
p_2 = x_3x_1 - b_2(x_4^2 - b_5 x_1^3 / 12) -
b_4(x_4x_3 - b_4 x_1^3/12), \\
 p_3 = x_4x_1 - b_4(x_4^2 - b_5
x_1^3 / 12) - b_5(x_4x_3 - b_4 x_1^3/12).
\end{gather*}
Finally let $H=y_4^3$.
If the first two lines of $M_b$ would be proportional,
a suitable linear transformation on $y_2$ and $y_3$ would yield $b_0=b_1=b_3=0$, whence $x_2\circ (y_1Q_b+H)=0$, thus $b\not\in B_H$.
 $\Ann(F^{H,b}) $ is generated by
 \begin{gather*}
12x_2^2-b_0x_1^3,\quad 12x_2x_3-b_1x_1^3,\quad 12x_3^2-b_2x_1^3,\quad
12x_4x_2-b_3x_1^3,\quad 12x_4x_3-b_4x_1^3,\\
b_6(x_4^2-
b_5x_1^3/12)-3x_4x_1(b_1^2-b_0b_2)+3x_3x_1(b_1b_3-b_0b_4)-3x_2x_1(b_2b_3-b_1b_4),\\
4x_4^3-x_1^4,\quad x_4^2x_3,\quad x_4^2x_2,\quad 12x_4^2x_1-b_5x_1^4,\quad x_4x_3^2,\quad
x_4x_3x_2,\quad 12x_4x_3x_1-b_4x_1^4,\\
 x_4x_2^2,\quad
12x_4x_2x_1-b_3x_1^4,\quad x_4x_1^2,\quad x_3^3,\quad x_3^2x_2,\quad
12x_3^2x_1-b_2x_1^4,\quad x_3x_2^2,\\
12x_3x_2x_1-b_1x_1^4, x_3x_1^2,\quad
x_2^3,\quad
12x_2^2x_1-b_0x_1^4, x_2x_1^2,
\end{gather*}
where $ b_6 = -\det(M_b).$
\end{proof}


\bigskip
\noindent
Gianfranco Casnati,\\
Dipartimento di Scienze Matematiche,  Politecnico di Torino, \\
corso Duca degli Abruzzi 24, 10129 Torino, Italy \\
e-mail: {\tt gianfranco.casnati@polito.it}

\bigskip
\noindent
Roberto Notari, \\
Dipartimento di Matematica \lq\lq Francesco Brioschi\rq\rq, Politecnico di Milano,\\
via Bonardi 9, 20133 Milano, Italy\\
e-mail: {\tt roberto.notari@polimi.it}

\enddocument